\documentclass[a4paper, 10pt]{article}
\usepackage[colorlinks=true]{hyperref}
\hypersetup{urlcolor=blue, citecolor=red}
\usepackage[T1]{fontenc}
\usepackage[utf8]{inputenc}
\usepackage{slashed}
\usepackage[normalem]{ulem}
\usepackage{cancel}
\usepackage{upgreek}
\usepackage{latexsym}
\usepackage{color}
\usepackage{amssymb,amsmath,amsfonts}
\newcommand\beq[1]{\begin{equation}\label{#1} }
\newcommand{\eeq}{\end{equation} }

\renewcommand{\theequation}{\arabic{section}.\arabic{equation}}

\newtheorem{theorem}{Theorem}
\newtheorem{definition}{Definition}[section]
\newtheorem{proposition}{Proposition}[section]
\newtheorem{lemma}{Lemma}[section]

\newtheorem{sublemma}{Sublemma}[section]
\newtheorem{remark}{Remark}[section]
\newtheorem{notationalremark}{Notational Remark}[section]

\newtheorem{corollary}[theorem]{Corollary}
\newtheorem{assumption}{Assumption}[section]
\newtheorem{claim}{Claim}[section]

\newtheorem{tools}{$\negsp\negsp$}[subsection]

\newcommand\thm[1]{\begin{theorem}\label{#1}}
\newcommand\thmtwo[2]{\begin{theorem}[#1]\label{#2}}
\newcommand\ethm{\end{theorem} }
\newcommand\dfn[1]{\begin{definition}\label{#1} \rm}
\newcommand\dfntwo[2]{\begin{definition}[#1]\label{#2} \rm}
\newcommand\edfn{\end{definition} }
\newcommand\pro[1]{\begin{proposition}\label{#1}}
\newcommand\protwo[2]{\begin{proposition}[#1]\label{#2}}
\newcommand\epro{\end{proposition} }
\newcommand\lem[1]{\begin{lemma}\label{#1}}
\newcommand\lemtwo[2]{\begin{lemma}[#1]\label{#2}}
\newcommand\elem{\end{lemma} }
\newcommand\sublem[1]{\begin{sublemma}\label{#1}}
\newcommand\sublemtwo[2]{\begin{sublemma}[#1]\label{#2}}
\newcommand\esublem{\end{sublemma} }
\newcommand\rem[1]{\begin{remark}\label{#1} \rm}
\newcommand\erem{\end{remark} }
\newcommand\notrem[1]{\begin{notationalremark}\label{#1} \rm}
\newcommand\enotrem{\end{notationalremark} }
\newcommand\cor[1]{\begin{corollary}\label{#1}}
\newcommand\cortwo[2]{\begin{corollary}[#1]\label{#2}}
\newcommand\ecor{\end{corollary} }
\newcommand\asmp[1]{\begin{assumption}\label{#1}}
\newcommand\asmptwo[2]{\begin{assumption}[#1]\label{#2}}
\newcommand\easmp{\end{assumption} }
\newcommand\clm[1]{\begin{claim}\label{#1}}
\newcommand\eclm{\end{claim} }
\newcommand{\proof}{\par\medskip\noindent{\bf Proof\ }}

\expandafter\chardef\csname pre amssym.def
at\endcsname=\the\catcode`\@
\catcode`\@=11
\def\undefine#1{\let#1\undefined}
\def\newsymbol#1#2#3#4#5{\let\next@\relax
 \ifnum#2=\@ne\let\next@\msafam@\else
 \ifnum#2=\tw@\let\next@\msbfam@\fi\fi
 \mathchardef#1="#3\next@#4#5}
\def\mathhexbox@#1#2#3{\relax
 \ifmmode\mathpalette{}{\m@th\mathchar"#1#2#3}%
 \else\leavevmode\hbox{$\m@th\mathchar"#1#2#3$}\fi}
\def\hexnumber@#1{\ifcase#1 0\or 1\or 2\or 3\or 4\or 5\or 6\or 7\or
8\or
 9\or A\or B\or C\or D\or E\or F\fi}
\ifcase\@ptsize
 \font\tenmsb=msbm10
 \font\sevenmsb=msbm7
 \font\fivemsb=msbm5
\or
 \font\tenmsb=msbm10 scaled \magstephalf
 \font\sevenmsb=msbm7 scaled \magstephalf
 \font\fivemsb=msbm5  scaled \magstephalf
\or
 \font\tenmsb=msbm10 scaled \magstep1
 \font\sevenmsb=msbm7 scaled \magstep1
 \font\fivemsb=msbm5 scaled \magstep1
\fi
\newfam\msbfam
\textfont\msbfam=\tenmsb
\scriptfont\msbfam=\sevenmsb
\scriptscriptfont\msbfam=\fivemsb
\edef\msbfam@{\hexnumber@\msbfam}
\def\Bbb#1{\fam\msbfam\relax#1}
\def\widehat#1{\setboxz@h{$\m@th#1$}%
 \ifdim\wdz@>\tw@ em\mathaccent"0\msbfam@5B{#1}%
 \else\mathaccent"0362{#1}\fi}
\def\widetilde#1{\setboxz@h{$\m@th#1$}%
 \ifdim\wdz@>\tw@ em\mathaccent"0\msbfam@5D{#1}%
 \else\mathaccent"0365{#1}\fi}

\def\RIfM@{\relax\ifmmode}
\def\nonmatherr@#1{\errmessage{\string#1\space allowed only in math mode}}
\def\Bbb{\RIfM@\expandafter\Bbb@\else
 \expandafter\nonmatherr@\expandafter\Bbb\fi}
\def\Bbb@#1{{\Bbb@@{#1}}}
\def\Bbb@@#1{\fam\msbfam\relax#1}
\def\setboxz@h{\setbox\z@\hbox}
\def\wdz@{\wd\z@}
\catcode`\@=\csname pre amssym.def at\endcsname
\newcommand{\negsp}{\hspace{-.09truecm}}  

\definecolor{yellow}{rgb}{0.99, 0.93, 0.0}

\begin{document}
\title{\bf No infinite spin for total collisions in the spatial N-body problem}
\date{\today}

\author{Gabriella Pinzari\footnote{Universit\`a degli Studi di Padova,
Dipartimento di Matematica Tullio Levi-Civita, Via Trieste 63 - 35121 Padova, Italy. E-mail: pinzari@math.unipd.it
}\ \  and Piotr Zgliczynski\footnote{Jagiellonian University, ul. $\slashed{L}$ojasiewicza 6, 30-348 Krak\'ov, Poland. E-mail: umzglicz@cyf-kr.edu.pl}}
\maketitle

\begin{abstract}\footnotesize{
}
In the $n$-body problem, when bodies tend to a total collision, then its normalized shape curve converges to the set of normalized central configurations, which has $SO(3)$
symmetry in the planar case. This leaves a possibility that the normalized  shape curve  tends to the set obtained by rotations of some central configuration instead of a particular point on it. This is the \emph{infinite spin problem} which concerns the rotational behavior of total collision orbits in the $n$-body problem. 
We show that the infinite spin is not possible if the limiting shape  is isolated from other connected components of the set of normalized central configurations. Our approach extends the method from recent work for total collision for the planar case by Moeckel and Montgomery.  The main tool is a full reduction $\rm SO(3)$--symmetry in a context of vanishing angular momentum.
\end{abstract}

\maketitle


\renewcommand{\theequation}{\arabic{equation}}

 \section{Introduction}
The spatial $(1+n)$--body problem is the $(6n+6)$--dimensional\footnote{ Here, by {\it dimension} of an ODE we mean the number of its equations, after it is reduced to a  first order ODE system. }  dynamical system
where $q=(q_1, \ldots, q_{n+1})\in{\mathbb R}^{3}\times \cdots\times {\mathbb R}^{3}\setminus \Delta$
solves
\begin{eqnarray}\label{motion equations}
\ddot q_i=\sum_{j=1\atop j\ne i}^{n+1}m_j\frac{q_j-q_i}{|q_i-q_j|^3}=\frac{1}{m_i} \nabla_{q_i} V(q) \qquad i=1\,,\ldots, n+1
\end{eqnarray}
where $|\cdot|$ denotes the Euclidean norm, the double dot  is the second derivative with respect to time,  $V(q)$ is the potential function 
\begin{eqnarray*} 
  V(q)=\sum_{j=1\atop j\ne i}^{n+1} \frac{m_i m_j}{r_{ij}}, \qquad r_{ij}=|q_i - q_j|
\end{eqnarray*}
 and 
\begin{eqnarray}\label{Delta}
\Delta:=\bigcup_{j=1\atop j\ne i}^{n+1}\Big\{q:\ q_i=q_j\Big\}
\end{eqnarray}
is the singular set.
By {\it total collision solution} we mean a solution  such that 
\begin{eqnarray}\label{total collision}T:=\inf\Big\{T_0>0:\ \lim_{t\to T_0^-}q_i(t)= B(T_0)\qquad \forall\ i=1\,,\ldots\,,n+1\Big\}<+\infty
\end{eqnarray}
where  $B(t)=B(0)+v_B t$ denotes  the uniform linear  motion of the center of mass $B$ $=$ $M^{-1}$ $\sum_{i=1}^{n+1}$ ${m_i q_i}$, with $M$ $=$  $\sum_{i=1}^{n+1}m_i$ the total mass. 
If this happens, then, clearly, condition~\eqref{Delta} is violated by all couples $(i, j)$ at time $T$.
For the spatial problem~\eqref{motion equations} there is an open question relating  total collisions to central configurations.    
We recall that a central configuration,  $\rm CC$,  is a value $q_{0}\in {\mathbb R}^{3(n+1)}$ such  that any particle has an acceleration pointing towards the center of mass, namely, solving 
$$\nabla_{q_i} V(q)=-\lambda  (q_i-B)\,,\qquad i=1\,,\ldots, n+1$$
with some $\lambda\in \mathbb R$.  
For any 
$q=(q_1, \ldots, q_{n+1})\in {\mathbb R}^{3}\times \cdots\times {\mathbb R}^3$, we define the \emph{normalized configuration}  $\widehat{q}=(\widehat q_1, \ldots, \widehat q_{n+1})$ by
\begin{equation} \label{normalized}
  \widehat{q}_i=\frac{q_i - B}{|q-BI_{n+1}|_m}, \quad i=1,\dots,n+1
\end{equation}
 where $I_{m}$ is the $m\times m$ unit matrix. 
Then $\widehat q=(\widehat q_1, \ldots, \widehat q_{n+1})\in {\mathcal S}$, with ${\mathcal S}$ is the unit sphere in $\mathbb R^{3(n+1)}$, relatively to the norm
$$|q|_m:=\sqrt{\sum_{i=1}^{n+1}m_i |q_i^2|}\,.$$
A result by Chazy states that, if $q(t)$ is  a total collision solution, then the $\upomega$--limit of $\widehat q(t)$ (namely, the set of $q_0$ such that there is a sequence of times $\{t_k\}\nearrow T$  such that $q(t_k)\to q_0$, as $k\to +\infty$) is a subset of the (normalized)   $\rm CC$.
By the $\rm SO(3)$--invariance of~\eqref{motion equations}, if $q_{0}\in {\rm CC}$, then  ${\mathcal R}q_{0}\in {\rm CC}$, for any ${\mathcal R}\in \rm SO(3)$, so we switch to the quotient spaces  ${\mathcal S}/\rm SO(3)$, ${\rm CC}/\rm SO(3)$, 
denoting as  $[\cdot]$ their respective elements.
We assume that the  $\upomega$--limit of $[\widehat q(t)]\in {\mathcal S}/\rm SO(3)$
contains  an isolated point $[q_{0}]\in {\rm CC}/\rm SO(3)$, with respect to the topology of ${\rm CC}/\rm SO(3)$. As ${\rm CC}/\rm SO(3)$ is compact, then such limit set is connected (see~\cite{hirschMD2013}), implying  that it actually has to coincide with  $[q_{0}]$. In other words, there is convergence for $[\widehat q(t)]$:
$$\lim_{t\to T^-} [\widehat q(t)]=[q_0]\,.$$
The natural question is: does  the 
corresponding 
normalized trajectory $\widehat q(t)$ defined through~\eqref{normalized} also converge (to a $\rm CC$), or is its $\upomega$-- limit set  some non--zero--measure subset of the set $${\mathcal G}(q_0):=\Big\{q_\star\in {\rm CC}:\ [q_\star]=[q_0]\Big\}=\Big\{q_\star={\mathcal R}q_0\,,\ {\mathcal R}\in \rm SO(3)\Big\}$$ like for example, a circle or a sphere?
The question has been posed by Wintner~\cite{wintner41}, also noted in~\cite{albouyCS2012}, and
solved in~\cite{moeckelM2025} for the planar problem, where convergence has been proved. Still for the planar case, the result has been then extended in~\cite{gierzkiewiczSZ2025,wangyou2026} so as to include the case of partial collisions, i.e., solutions such that~\eqref{total collision} holds only for a cluster of indices $i$.   Here we extend the result to the spatial problem. We however need an additional  condition.
A  total collision solution $q(t)$ of~\eqref{motion equations} will be said  {\it non--collinear} if there exist  $1\le i<j\le n$ such that the  Jacobi coordinates (see the next section) 
  verify   
\begin{eqnarray}\label{non collinearity}
\sup_{0\le t<T}
\frac{|x_i(t)\cdot x_j(t)|}{|x_i(t)\times x_j(t)|}\,,\qquad \sup_{0\le t<T}\frac{|x_j(t)|^2}{|x_i(t)\times x_j(t)|}<+\infty\end{eqnarray}
In particular, the particles are not allowed to go to total collision in a collinear way.
We prove the following result, which will be stated with more precision in the course of the paper (see Theorem~\ref{main} below).
  
\vskip.1in
\noindent
{\bf Theorem A} {\it Suppose $q(t)$
is a non--collinear, total collision solution of the spatial $(n+1)$--body problem. Suppose  $[\widehat q(t)]$ converges to an isolated $[q_0]\in {\rm CC}/\rm SO(3)$. If $$(\varphi, \theta, \psi)\in {\mathbb T}\times (0, \pi)\times {\mathbb T}\,,$$ 
where ${\mathbb T}:={\mathbb R}/(2\pi\mathbb Z)$, are three Euler angles  reducing the $\rm SO(3)$--symmetry of the translationally reduced system coming from~\eqref{motion equations}, then $(\varphi(t), \theta(t), \psi(t))$ converge as $t\to T^-$ and so $\widehat q(t)$ converges to a particular $q_\star\in {\mathcal G}(q_0)\subset{\rm CC}$.}

\vskip.1in
\noindent
To highlight our contribution, we recall in a very sketchy way the strategy followed in~\cite{moeckelM2025}. 
It is based on various ingredients: reduction of translations, of rotations, McGehee blow--up, analysis of the flow on the center manifold and its stable fibers. The two former tools are finalized to reduce the dimension of the  planar version of the  system~\eqref{motion equations} 
from  $4n+4$ to $4n$ and next to $4n-1$. 
The procedure  allows  also to switch to new coordinates, which here we name $(\upxi, \theta)$, where 
$\theta\in \mathbb T$ is an angle which reduces $\rm SO(2)$--invariance, while $\upxi\in \mathbb R^{4n-2}$ are residual coordinates in $\mathbb R^{4n}/{\rm SO(2)}$. The new coordinates solve an ODE of the form 
 \begin{eqnarray}\label{ODE}
\upxi'=g(\upxi)\,,\qquad 
\theta'=f(\upxi, \upxi') 
\end{eqnarray}
where the prime denotes the derivative with respect to a new time $\tau$, with $t\to T^-$ corresponding to $\tau\to \infty$. 
  The reduction is visible in the new equations~\eqref{ODE}, because the motion of the $\upxi$'s does not depend on the one of $\theta$, while the evolution of the latter coordinate is found integrating the latter equation, once the former is solved. Important technical facts in the resulting equations~\eqref{ODE} are: 
(i) $f$ is linear with respect to $\upxi'$, and
 (ii) isolated equilibria for the vector--field $g$ correspond to isolated classes of limit central configurations for 
$[\widehat q(t)]$, in the old coordinates.
Then the analysis reduces to studying convergence of the $\upxi$--coordinates to an isolated equilibrium through their  center/stable manifolds.  The case when the center manifold is absent (hence, the equilibrium is hyperbolic) is by far easier and has been previously treated in~\cite{chazy1918}. Indeed, in such case $\upxi$ converges exponentially fast to the equilibrium, implying that $|\upxi'(\tau)|$ has a finite integral over $[0, +\infty)$. As  $f$ in~\eqref{ODE} is linear with respect to $\upxi'$, no--infinite spin follows from a bound like

\begin{eqnarray*}
\lim_{\tau\to +\infty}|\theta(\tau)|
&\le&|\theta(0)|+\lim_{\tau\to +\infty}|\theta(\tau)-\theta(0)|
\le|\theta(0)|+\int_0^{+\infty}|\theta'(\tau)| d\tau\nonumber\\
&\le& |\theta(0)|+C\int_{0}^{+\infty}|\upxi'(\tau)|<+\infty\,.
\end{eqnarray*}
The paper~\cite{moeckelM2025} deals with the analysis of the case when the center manifold is present, where the authors still prove convergence for the right hand side integral.  \\
   Here we basically follow the same scheme. We begin with the reduction of 
translation invariance due to Jacobi.   This  is a standard step, which transforms the ODE~\eqref{motion equations} to a similar one for a fictitious $n$--particle system, hence lowers the dimension to $6n$.
From now on, we refer only to the Jacobi translationally-reduced system, which, as well as the original one, 
is still $\rm SO(3)$--invariant.
Our main contribution relies in the use of a reduction of this symmetry,
    in a framework where the  total angular momentum  (of the translationally reduced system) -- $\rm C$ in what follows -- vanishes. 
As (contrarily to the ones in the plane) rotations in the space do not 
commute,  reducing the $\rm SO(3)$--symmetry is a non--trivial task going back to Jacobi~\cite{jacobi1842}, who, in the case $n=2$ (three--body problem) and under condition $\rm C\ne 0$, obtained a reduction of the  dimension of the system~\eqref{motion equations}, known as {\it nodes reduction},  by producing new coordinates  obeying to an ODE again of the form~\eqref{ODE}, but with $\upxi\in {\mathbb R}^{8}$.  Note that this corresponds to a reduction of the dimension by {\it four} units, as the translationally reduced system has dimension $6n=12$. 
We remark  that condition $\rm C\ne 0$ in Jacobi's work is not
a merely technical fact. Rather, it is an assumption without which the procedure looses its meaning, because  
it uses, as a main tool,  a  reference frame ({\it invariable frame})
where the third axis is parallel to the direction of  $\rm C$. Furthermore,  the angle $\theta$,  main actor of the dimensional reduction,  returns the anomalies of the  mutually opposed, straight lines ({\it nodes}) obtained by intersecting 
the orbital planes\footnote{ By orbital plane of a particle in $\mathbb R^3$ we mean the instantaneous plane determined by  its Cartesian coordinates (with respect to a prefixed reference frame) and the ones of its velocity.} of the two fictitious particles with
the plane orthogonal to $\rm C$. 
Later,  Radau~\cite{radau1868}, interpreted Jacobi's reduction in a Hamiltonian--like language,  succeeding to write the
former ODE in~\eqref{ODE} as the Hamilton equations of a suitably related  Hamiltonian.
For a considerable period, Jacobi-Radau reduction seemed to be destined to remain an isolated, fortuitous occurrence. The opposition of  the nodes was believed, in the community,  to have no chance to be replicated when $n\ge 3$. The lack of such an useful tool for the general case was sadly mentioned by V. I. Arnold in~\cite{arnold63b} as an obstruction in order to bypass the problem of KAM degeneracy
in the framework of the proof of stability of planetary motions.
It was only one century and one half later, in 1982, that a young Ph.D student, F. Boigey~\cite{boigey82}, found a generalization to the case of $n=3$. In this case the phase space has dimension $6n=18$.
In an entirely Hamiltonian framework, Boigey produced a $18$--dimensional  set of canonical coordinates $(p, q)\in {\mathbb R}^9\times{\mathbb R}^9$, where three of them, say $p_{8}$, $p_9$, $q_9$, were obtained  as functions of the three components of the angular momentum ${\rm C}$. This allowed her to disregard the constant motion of such coordinates and write, for the remaining ones, equations again of the form~\eqref{ODE}, for  $\upxi=(p_1, \ldots, p_7, q_1, \ldots, q_7)\in {\mathbb R}^{14}$ and $\theta=q_8$, parametrically depending on $p_8=|\rm C|$. Note that 
the couple of coordinates $(p_9, q_9)$ does not appear in the Hamiltonian, being both first integrals, hence, both cyclic.
In the Hamiltonian language, the dimension reduction corresponds to the production of three cyclic coordinates, namely, $q_8=\theta$ and
the aforementioned couple $(p_9, q_9)$.
It is precisely this couple of canonical coordinates which marks the
advancement of knowledge with respect to Radau--Jacobi's points of view, because (in a suitable sense) 
it  gives the (negligible) direction of $\rm C$. Also Boigey's reduction works under condition  $\rm C\ne 0$ (plus some others), by using a chain of reference frames, whose  third axes are parallel to partial sums of angular momenta, including $\rm C$.
The complete generalization to any $n$ arrived one year later, by A. Deprit~\cite{deprit83}. Generalizing (and acknowledging) Boigey, Deprit constructed a set of $6n$--dimensional canonical coordinates $(p, q)\in {\mathbb R}^{3n}\times{\mathbb R}^{3n}$, where
 $p_{3n-1}$, $q_{3n-1}$,  $p_{3n}$  and $q_{3n}$ play the r\^ole of Boigey's  $p_{8}$, $q_{8}$, $p_{9}$  and $q_{9}$. Maybe due to the 
complicated geometric features underlying Deprit's reduction, which led also the author himself to express perplexities about the concrete benefit of his discovery, Deprit's reduction was overlooked for about 30 years. 
In~\cite{chierchiaPi11b} it has been used to complete the aforementioned Arnold's program, even though the proof of KAM stability of planetary motions had already appeared in
~\cite{fejoz04}, thanks to a tricky argument to circumvent the mentioned degeneracy. 
Recently, Deprit's reduction has been used  in~\cite{clarkeFG2024} to prove Arnold instability of semi--major axes in a four--body problem.
Other reductions appeared in~\cite{pinzari18} (see~\cite{pinzari2024} for a review), aimed to bypass the singularity of Deprit's reduction when some orbital plane tend to coincide, an aspect of interest in planetary systems.
It is worth  remarking that the dimensional reduction by four units shared by all the aforementioned procedures  is optimal and is intimately related to the non--commutativity of rotations in $\mathbb R^3$. 
Indeed, a classical  result  by Liouville~\cite{liouville1853} states that  if a system has $m$ commuting and independent first integrals, then,  under suitable conditions, the dimension of phase space can be lowered by $2m$. As the largest number of commuting  first integrals that one can form with the three components of $\rm C$ is {\it two}, we then have that {\it four} is the maximal dimensional reduction that one can reach.
\\
   Notwithstanding the variety of the above mentioned results, none of them applies  to the problem considered in the paper, because, as anticipated, a total collision solution can happen only if the total angular momentum vanishes  (see~\cite{sundman13,wintner41}), a case when all of them crash. At our knowledge, canonical $\rm SO(3)$ reductions with $\rm C=0$ have not been explored in the literature.    We propose a canonical change of coordinates which, generally speaking, does not completely reduce  the $\rm SO(3)$ symmetry of the system, because just {\it two} units  go down. However, when such change is applied to systems with vanishing angular momentum, the  dimension  reduction is higher. More precisely, we again obtain equations of the form~\eqref{ODE}, but with  $\upxi\in {\mathbb R}^{6n-6}$ as coordinates in ${\mathbb R}^{6n}/\rm SO(3)$, while, in the latter, $\theta$ is replaced with the triple of Euler angles $(\varphi, \theta, \psi)$ mentioned in the statement. This picture marks two notable differences with the case $\rm C\ne0$. First of all, the dimension of the reduced space ${\mathbb R}^{6n}/\rm SO(3)$ (namely, of the coordinates $\upxi$) is $6n-6$, compared to the value $6n-4$ that we have when $\rm C\ne0$. 
   This is not a contradiction, because  when $\rm C=0$, the well known commutation rules among its components
  \begin{eqnarray*}\{{\rm C}^i, {\rm C}^j\}=-\delta_{ijk}{\rm C}^k
\end{eqnarray*}
(where  $\{\cdot\,,\cdot\}$ are the Poisson brackets, 
$(ijk)$ is a permutation of $(123)$ and $\delta\in\{\pm 1\}$ is its parity)
 have zero at the right hand sides, which, roughly, suggests one should be able to take $m=3$ in Liouville theorem.
   Secondly, 
when $\rm C=0$, 
all the angles in Euler triple move. This is opposite to what happens when $\rm C\ne 0$, because in such case the conservation of $\rm C$ imposes that the angles $\varphi$ and $\theta$ -- which give the direction of the plane orthogonal to $\rm C$ --
stay at rest.  We prove that, when $\rm C=0$, all the three moving Euler angles
have a finite limit as $t\to T^-$.
\section{Reduction of translations} \label{sec:red-trans}
We proceed in the Hamiltonian framework. The Hamiltonian of the $(1+n)$--body problem in $\mathbb R^3$ is
 \begin{eqnarray}\label{H}
H_{n+1}=\sum_{i=1}^{n+1}\frac{|p_i|^2}{2m_i}-\sum_{1\le i<j\le n+1}\frac{m_i m_j}{|q_i-q_j|}
\end{eqnarray}
where $p=(p_1, \ldots, p_{n+1})\in{\mathbb R}^3\times \cdots\times {\mathbb R}^3$, $q=(q_1, \ldots, q_{n+1})\in{\mathbb R}^3\times \cdots\times {\mathbb R}^3\setminus\Delta$, with $\Delta$ as in~\eqref{Delta}.
 As well known, there exists a canonical,
 linear  change of coordinates
\begin{eqnarray}\label{reduction of translations}(P, y_1, \ldots, y_n, B, x_1, \ldots, x_n)\to (p, q)
\end{eqnarray}
which splits
  $H_{n+1}(p, q)$ as
\begin{eqnarray*}
H_{n+1}(P, y, B, x)=H_{B}(P)+ H(y,  x)
\end{eqnarray*}
 with $y=(y_1, \ldots, y_{n})$, $ x=(x_1, \ldots, x_{n})$, with
 $H_{B}(P)$ an unessential term,
 while 
\begin{eqnarray}\label{HNEW} H( y, x)=\sum_{1\le i\le n}\frac{|y_i|^2}{2\mu_i}-\sum_{0\le i<j\le n}\frac{m_{i+1}m_{j+1}}{\left|-\frac{M_{i}}{M_{i+1}}x_{i}+
 \sum_{k=i+1}^{j-1}\frac{m_{k+1}}{M_{k+1}}x_k+
 x_{j}
 \right|} \end{eqnarray}
 with $\mu_i$ suitable new mass parameters depending on $m_1$, $\ldots$, $m_{n+1}$, given in~\eqref{mui} below.
%
%
  Such transformation is   often referred to as {\it Jacobi reduction of translations} and is the unique linear,   exact  and symplectic transformation defined so that, if $M:=\sum_{i=1}^{n+1}m_i$ is the total mass,  $B:=M^{-1}\sum_{i=1}^{n+1}{m_i q_i}$ is the center of mass of the whole system; $x_i$ is the relative distance of the particle with mass $m_{i+1}$  to the center of mass of $m_{1}$, $\ldots$, $m_{i}$;  $P$, $y_i$ are the  impulses conjugated to $B$, $x_i$, respectively.  
More details  are reported in Appendix~\ref{Barycentric reduction (generalized Jacobi coordinates)}.

\section{Reduction of rotations} \label{sec:rot-red}

The Hamiltonian~\eqref{HNEW} is still $\rm SO(3)$--invariant, as well as the original one~\eqref{H}. This follows from the fact that the angular momentum retains its form after switching to Jacobi coordinates. Namely, the following relation holds (see Appendix~\ref{Barycentric reduction (generalized Jacobi coordinates)})
\begin{eqnarray}\label{eq: ang mom}
\sum_{i=1}^{n+1} q_i\times p_i=B\times P+\sum_{i=1}^n x_i\times y_i
\end{eqnarray}
where ``$\times$'' denotes skew--product. In this section, we face the problem of ruling this symmetry out.

\vskip.1in
\noindent
In what follows, the dot ``$\cdot$'' will denote the Euclidean inner product of two vectors in $\mathbb R^3$, $(\cdot\,,\cdot)$ the corresponding operation in the coordinate space.  Moreover, we denote as  ${\mathbb R}_+:=\{x\in {\mathbb R}:\ x\ge 0\}$; ${\mathbb R}_-:=\{x\in {\mathbb R}:\ x\le 0\}$ and, finally, $\mathbb T:={\mathbb R}/(2\pi\mathbb Z)$. 

\vskip.1in
\noindent
We consider
$n$   pairs of triples of canonical variables $(y_j, x_j)\in \mathbb R^3\times \mathbb R^3$, with $j=1$, $\ldots$, $n$, $x_j=(x_{j, 1}, x_{j, 2}, x_{j, 3})$, $y_j=(y_{j, 1}, y_{j, 2}, y_{j, 3})$.
We regard $x_1$, $\ldots$, $x_n$ as if they were the Cartesian coordinates of $n$ particles  located at the points $Q_1$, $\ldots$, $Q_n\in\mathbb R^3$,
relatively to  a fixed reference frame $Oe_1e_2e_3$ in $\mathbb R^3$, and $y_1$, $\cdots$, $y_n$ the coordinates of the respective linear momenta $P_1$, $\ldots$, $P_n\in\mathbb R^3$.  From now on, we  identify
$(Q_1$, $\ldots$, $Q_n)$ and $(P_1$, $\ldots$, $P_n)$
with their coordinate triples $(x_1$, $\ldots$, $x_n)$, $(y_1$, $\ldots$, $y_n)$ relatively to $Oe_1e_2e_3$.
We consider a new, moving, reference frame, $O{f}_1{f}_2{f}_3$ , where 
\begin{eqnarray}\label{eq:f1f2f3}
{f}_3=\frac{x_n}{|x_n|}\,,\qquad {f}_1=\frac{x_{n}\times x_{n-1}}{|x_{n}\times x_{n-1}|}\,,\qquad {f}_2={f}_3\times {f}_1\,.
\end{eqnarray} 
The frame $O{f}_1{f}_2{f}_3$ is well defined provided that  
\begin{eqnarray}\label{existence conditions}
x_n\ne 0\,,\qquad x_{n-1}\not\parallel x_n\,.
\end{eqnarray}
In the application in this paper, 
 conditions will be weakened to~\eqref{non collinearity} (provided that $i$, $j$ are re--labeled as $n-1$, $n$, respectively.).
\\
We denote as $\varphi$, $\theta$ and $\psi$  the precession, nutation and proper rotation angles relatively to the frames $Oe_1e_2e_3$ and $O{f}_1{f}_2{f}_3$, defined as follows. Given $u$, $v$, $w\in \mathbb R^3$, with $w\perp u$, $v$, let $\alpha_w(u, v)$ stand for the angle  which $u$ has to run, positively (counterclockwise with respect to $w$), to overlap its direction  and verse to the ones of $v$. Then
\begin{eqnarray*}
\varphi:=\alpha_{e_3}(e_1, \gamma)\,,\qquad \theta:=\alpha_{\gamma}(e_3, {f}_3)\,,\quad \psi:=\alpha_{{f}_3}(\gamma, {f}_1)
\end{eqnarray*}
with  $\gamma:=\frac{e_3\times {f}_3}{|e_3\times {f}_3|}$.
and
$$\varphi\,,\psi\in {\mathbb T}\,,\qquad \theta \in (0, \pi)\,.$$
The triple $(\varphi, \theta, \psi)$ is well defined provided that
\begin{eqnarray}\label{existence conditionsNEW}
x_n\not\parallel e_3\,.
\end{eqnarray}
We denote as $\upxi_j=(\upxi_{j, 1}, \upxi_{j, 2}, \upxi_{j, 3})$, $j=1$, $\ldots$, $n$, the triples of  cartesian coordinates  of $x_j$ relatively to $O{f}_1{f}_2{f}_3$  
and put $\upxi=(\upxi_1\,,\ldots\,,\upxi_n)$. 
By definition,   $\upxi_{j, k}$ are defined through
\begin{eqnarray}\label{xn xn-1}
x_j=\left\{
\begin{array}{lll}
\displaystyle\upxi_{j, 1}f_1+\upxi_{j, 2}f_2+\upxi_{j, 3}f_3\qquad &j=1\,,\ldots\,,n-2\\\\
\displaystyle\upxi_{n-1, 2}f_2+\upxi_{n-1, 3}f_3& j=n-1\\\\
\upxi_{n, 3}f_3&j=n
\end{array}
\right.
\end{eqnarray}
where  the numbers $\upxi_{n, 3}$ and $\upxi_{n-1, 2}$
satisfy

\begin{eqnarray}
\upxi_{n, 3}\in \mathbb R_+\setminus\{0\}\,,\qquad \upxi_{n-1, 2}\in \mathbb R_-\setminus\{0\}\,.  \label{eq:upxi-ranges}
\end{eqnarray}
The particular decompositions~\eqref{xn xn-1} of $x_{n-1}$, $x_n$ as well as the inequalities~\eqref{eq:upxi-ranges} are immediate consequences of the definitions
\eqref{eq:f1f2f3} and conditions~\eqref{existence conditions}.
We now define the momenta $(\Phi, \Theta, \Psi, \upeta)$, respectively conjugated to $(\varphi, \theta, \psi, \upxi)$.
We denote as 
\begin{eqnarray*}
{\rm C}:=\sum_{j=1}^n x_j\times y_j\,.\end{eqnarray*}
If the coordinates $(y, x)$ are chosen as in Section~\ref{sec:red-trans},
$\rm C$ is the angular momentum of the translationally reduced Hamiltonian~\eqref{HNEW}, already appeared in  the last summand in~\eqref{eq: ang mom}.
Then we let
\begin{eqnarray}\label{ThetaPhi}\Phi:={\rm C}\cdot e_3\,,\qquad \Theta:={\rm C}\cdot \gamma\,,\qquad \Psi:={\rm C}\cdot {f}_3 \end{eqnarray}
Finally, we denote as $\upzeta_j=(\upzeta_{j, 1}, \upzeta_{j, 2}, \upzeta_{j, 3})$, $j=1$, $\ldots$, $n$, 
the coordinates of $y_j$ relatively to $Of_1f_2f_3$.
Beware that the $\upzeta_{j}$'s do not have  the same structure as the $\upxi_{j}$'s  in~\eqref{xn xn-1}, but, more generally,
\begin{eqnarray}\label{upzetas}y_j=\upzeta_{j, 1}f_1+\upzeta_{j, 2}f_2+\upzeta_{j, 3}f_3\qquad j=1\,,\ldots\,,n\,.\end{eqnarray}
However, we let $\upeta=(\upeta_1\,,\ldots\,,\upeta_n)$, with
\begin{eqnarray}\label{upetas}\upeta_j:=\left\{
\begin{array}{lll}
\displaystyle\upzeta_{j, 1}f_1+\upzeta_{j, 2}f_2+\upzeta_{j, 3}f_3\qquad &j=1\,,\ldots\,,n-2\\\\
\displaystyle\upzeta_{n-1, 2}f_2+\upzeta_{n-1, 3}f_3 & j=n-1\\\\
\upzeta_{n, 3}f_3&j=n
\end{array}
\right.
\end{eqnarray}
Then put
\begin{eqnarray}\label{p}
\left\{
\begin{array}{lll}
P:=\left(\Phi\,,\Theta, \Psi,\upeta\right)\in
{\mathbb R}\times{\mathbb R}\times{\mathbb R}\times({\mathbb R}^3)^{n-2}\times{\mathbb R}^2\times{\mathbb R}\\\\
Q=\big(\varphi, \theta, \psi, \upxi\big)\in
{\mathbb T}\times(0, \pi)\times{\mathbb T}\times({\mathbb R}^3)^{n-2}\times{\mathbb R}\times({\mathbb R}_-\setminus\{0\}) \times({\mathbb R}_+\setminus\{0\})\,.
\end{array}
\right.
\end{eqnarray}
and take in mind that 
the components $\upzeta_{n, 1}$, $\upzeta_{n, 2}$ and $\upzeta_{n-1, 1}$ 
appearing in~\eqref{upzetas} 
but not in~\eqref{upetas} will be functions 
(explicitly given in Equations~\eqref{etas} below)
of  $(P, Q)=(\Phi$, $\Theta$, $\Psi$, $\upeta$, $\varphi$, $\theta$, $\psi$, $\upxi)$. 
We denote as
 \begin{eqnarray}\label{Phi}
 \phi:\quad (P, Q)\to (y, x)
 \end{eqnarray}
the map which relates the coordinates $(P, Q)$ defined above with $(y, x)$.

\begin{remark}\rm
If the coordinates~\eqref{p} are used in a $\rm SO(3)$ invariant Hamiltonian, like~\eqref{HNEW}, one generally expects a  dimensional reduction  by just  two units, because only the coordinate $\varphi$ is cyclic, and one can look at
the Hamiltonian equations of~\eqref{HNEW} for all the coordinates $(P, Q)$ in~\eqref{p}, but $(\Phi, \varphi)$, regarding $\Phi$ as an external parameter. However, 
as we shall discuss in Section~\ref{sec:red-ham},  when $\rm C=0$, 
one has a dimensional reduction by six units,
 as anticipated in the introduction.
\end{remark}
\begin{proposition}\label{homogeneous-canonical}The map $\phi$ in~\eqref{Phi} is exact symplectic.  
\end{proposition}
To prove Proposition~\ref{homogeneous-canonical}, we represent the map~\eqref{Phi}. 
As $\upxi_j$, $\upzeta_j$ are the coordinates of $x_j$, $y_j$ relatively to $Of_1f_2f_3$, we have that the map $\phi$ in~\eqref{Phi} can be represented as
\begin{eqnarray}\label{phi map}
\phi:\quad	\left\{
	\begin{array}{lll}
		x_j={\mathcal R}(\varphi, \theta, \psi)
		\upxi_j
		\\\\
		y_j={\mathcal R}(\varphi, \theta, \psi)\upzeta_j
	\end{array}
	\right.
	\qquad j=1\,,\ldots\,, n
\end{eqnarray}
with
\begin{eqnarray*}
{\mathcal R}(\varphi, \theta, \psi):={\mathcal R}_3\left(\varphi\right){\mathcal R}_1(\theta){\mathcal R}_3(\psi)
\end{eqnarray*}

where 
\begin{eqnarray*}
{\cal R}_1(\alpha)&=&\left(
\begin{array}{ccc}
1&0&0\\
0&\cos\alpha&-\sin\alpha\\
0&\sin\alpha&\cos\alpha
\end{array}
\right)\,,\quad  {\cal R}_3(\alpha)=\left(
\begin{array}{ccc}
\cos\alpha&-\sin\alpha&0\\
\sin\alpha&\cos\alpha&0\\
0&0&1
\end{array}
\right)\end{eqnarray*}
(see Appendix~\ref{sec:Euler-angles}). Moreover, we let
\begin{eqnarray*}
e_1^*=
\left(
\begin{array}{ccc}
1\\
0\\
0
\end{array}
\right)\,,\quad 
e_3^*=
\left(
\begin{array}{ccc}
0\\
0\\
1
\end{array}
\right)
\end{eqnarray*}

\begin{lemma}[\cite{pinzari18}]\label{lem:trivial} Let  $j\in \{1\,,  3\}$, $y$, $x$, $\upzeta$, $\upxi\in \mathbb R^3$;
$y={\mathcal R}_j(\alpha) \upzeta$, $x={\mathcal R}_j(\alpha) \upxi$.
Then
$$y\cdot dx=\left({\rm C}\,,e_j^*\right)d\alpha+\upzeta\cdot d\upxi$$
where ${\rm C}:=x\times y$.
\end{lemma}

\proof Let us denote $J_j=\frac{d}{d\alpha}\mathcal{R}_j(\alpha)_{|\alpha =0}$. Then
\begin{eqnarray*}
   \frac{d}{d\alpha}\mathcal{R}_j(\alpha)&=& J_j \mathcal{R}_j(\alpha) = \mathcal{R}_j(\alpha) J_j \\
   y \cdot (J_j x)&=& (J_j x) \cdot y =  (x \times y)_j = (x \times y) \cdot e_j^*
\end{eqnarray*}
We have 
\begin{eqnarray*}
   y\cdot dx &=& (\mathcal{R}_j(\alpha)\upzeta) \cdot d(\mathcal{R}_j(\alpha)\upxi)=(\mathcal{R}_j(\alpha)\upzeta) \cdot (\mathcal{R}_j(\alpha)J_j \upxi d\alpha) 
   + (\mathcal{R}_j(\alpha)\upzeta) \cdot (\mathcal{R}_j(\alpha)d \upxi)  \\
 &=&  (\upzeta \cdot (J_j \upxi )) d\alpha +  \upzeta \cdot d \upxi = (\upxi \times \upzeta) \cdot e_j^*  +  \upzeta \cdot d \upxi.
\end{eqnarray*}
Observe that 
\begin{eqnarray*}
  (x \times y) \cdot e_j^* = (({\mathcal R}_j(\alpha) \upxi) \times ({\mathcal R}_j(\alpha) \upzeta))\cdot (\mathcal{R}_j e_j^*) =  ({\mathcal R}_j(\alpha) (\upxi \times \upzeta)) \cdot (\mathcal{R}_j e_j^*)
  = (\upxi \times \upzeta) \cdot e_j^* \,.\qquad \square
\end{eqnarray*}

\proof {\bf of Proposition~\ref{homogeneous-canonical}}
{ Let   $x_j={\mathcal R}_3(\varphi){\mathcal R}_1(\theta){\mathcal R}_3(\psi)\upxi_j$ and $y_j={\mathcal R}_3(\varphi){\mathcal R}_1(\theta){\mathcal R}_3(\psi)\upzeta_j$. Applying Lemma~\ref{lem:trivial} iteratively we obtain
\begin{eqnarray*}
  y_j \cdot dx_j &=& (x_j \times y_j, e_3^*) d\varphi + ({\mathcal R}_1(\theta){\mathcal R}_3(\psi)\upzeta_j) \cdot d ({\mathcal R}_1(\theta){\mathcal R}_3(\psi)\upxi_j) \\
   &=&  (x_j \times y_j, e_3^*) d\varphi   + (({\mathcal R}_1(\theta){\mathcal R}_3(\psi)\upxi_j) \times ({\mathcal R}_1(\theta){\mathcal R}_3(\psi)\upzeta_j), e_1^*) d \theta \\
   &+& ({\mathcal R}_3(\psi)\upzeta_j) \cdot d({\mathcal R}_3(\psi)\upxi_j)  \\
    &=& (x_j \times y_j, e_3^*) d\varphi   + (({\mathcal R}_3(\varphi){\mathcal R}_1(\theta){\mathcal R}_3(\psi)\upxi_j) \times ({\mathcal R}_3(\varphi){\mathcal R}_1(\theta){\mathcal R}_3(\psi)\upzeta_j), {\mathcal R}_3(\varphi)e_1^*) d \theta  \\
     &+& ((\mathcal{R}_3(\psi)\upxi_j)\times (\mathcal{R}_3(\psi)\upzeta_j),e^*_3) + \upzeta_j \cdot d\upxi_j  \\
 &=&   (x_j \times y_j, e_3^*) d\varphi   + (x_j \times y_j, {\mathcal R}_3(\varphi)e_1^*) d \theta  +  (x_j \times y_j, {\mathcal R}_3(\varphi) \mathcal{R}_1(\theta)e_1^*) d \psi + \upzeta_j \cdot d\upxi_j.
\end{eqnarray*}
}
Summing over $j$,  and using~\eqref{xn xn-1}, we have
\begin{eqnarray*}\sum_{j=1}^ny_j\cdot dx_j&=&  \left({\rm C}\,, e^*_3\right)\, d\varphi+\left({\rm C}\,, {\mathcal R}_3\left(\varphi\right)e^*_1\right) d\theta+\left({\rm C}\,, {\mathcal R}_3\left(\varphi\right){\mathcal R}_1(\theta)e^*_3\right) d\psi
+\sum_{j=1}^{n-2}\upzeta_j\cdot d\upxi_j\nonumber\\
&+&\sum_{k\in\{2\,,3\}}\big(\upzeta_{n-1}\,, e_k^*\big) d\upxi_{n-1, k}+
\big(\upzeta_{n}\,,e_3^*) d\upxi_{n, 3}
\ \nonumber\\
&=&\Phi d\varphi+\Theta d\theta+\Psi d\psi +\sum_{j=1}^{n-2}\upzeta_j\cdot d\upxi_j+\sum_{k\in\{2\,,3\}}\upzeta_{n-1, k} d\upxi_{n-1, k}+
\upzeta_{n, 3} d\upxi_{n, 3}\nonumber\\
&=&\Phi d\varphi+\Theta d\theta+\Psi d\psi +\upeta\cdot d\upxi
\end{eqnarray*}
having  recognized that (see formula (\ref{eq:appR3R1e3f3}) in Lemma~\ref{lem:EulerAngles})
\begin{eqnarray*}
&&{\mathcal R}_3\left(\varphi\right)e^*_1=\gamma\,,\qquad {\mathcal R}_3\left(\varphi\right){\mathcal R}_1\left(\theta\right)e^*_3={f}_3\,,\qquad \left(\upzeta_n\,, e^*_3\right)=
\upzeta_{n, 3}
\end{eqnarray*}
and, similarly,
$$\left(\upzeta_{n-1}\,, e^*_2\right)=\upzeta_{n-1, 2}
\,,\qquad \left(\upzeta_{n-1}\,, e^*_3\right)=\upzeta_{n-1, 3}
$$
and having used the definitions of $\Theta$, $\Psi$, $\Phi$ and $\upeta$. $\qquad \square$

\section{The reduced Hamiltonian}
\label{sec:red-ham}

 In this section, we assume conditions~\eqref{existence conditions} and~\eqref{existence conditionsNEW}, which, in terms of the coordinates $(P, Q)$ in~\eqref{p}, are expressed through~\eqref{eq:upxi-ranges} and
  $$ \theta\in (0, \pi)\,.$$
We 
%
%
%
 proceed to write the Hamiltonian~\eqref{HNEW} in terms of the coordinates~\eqref{p}. As first thing, we need the expressions 
of $\zeta_{n-1, 1}$, $\zeta_{n, 1}$ and $\zeta_{n, 2}$ in terms of 
$(P, Q)$. 
In Section~\ref{subsec:Proof-of-etas} we shall check that the these are given by
\begin{eqnarray}\label{etas}\left\{
\begin{array}{lll}
\displaystyle\upzeta_{n-1, 1}=-\frac{1}{\upxi_{n-1, 2}}(\Psi-\Psi^{(n-2)}_3(\upeta, \upxi))\\\\
\displaystyle\upzeta_{n, 1}=\frac{1}{\upxi_{n, 3}}\left(-\Theta\sin\psi+\frac{\Phi-\Psi\cos\theta}{\sin\theta}\cos\psi-\Psi_2^{(n-2)}(\upeta, \upxi)+\frac{\upxi_{n-1, 3}}{\upxi_{n-1, 2}}(\Psi-\Psi^{(n-2)}_3(\upeta, \upxi))\right)\\\\
\displaystyle\upzeta_{n, 2}=\frac{1}{\upxi_{n, 3}}\left(-\Theta\cos\psi-\frac{\Phi-\Psi\cos\theta}{\sin\theta}\sin\psi+\Psi_1^{(n-1)}(\upeta, \upxi)\right)
\end{array}\right.
\end{eqnarray}
with
\begin{eqnarray}
\label{Psis}\Psi^{(p)}_k(\upeta, \upxi):=
\sum_{j=1}^p \big(\upxi_{j} \times \upeta_{j},e_k^*\big)
\end{eqnarray}
%
%
%
%
%
By
plugging the formulae~\eqref{phi map} and~\eqref{etas} into
\eqref{HNEW}, we obtain the expression for the Hamiltonian
\begin{eqnarray}&& H(\Phi, \Theta, \Psi, \upeta, 
\theta, \psi, \upxi)=\sum_{j=1}^{n}\frac{|\upeta_j|^2}{2\mu_j}
	+
	\frac{(\Psi-\Psi_{3}^{(n-2)}(\upeta, \upxi))^2}{2\mu_{n-1}\upxi_{n-1, 2}^2} \label{HNEWNEW}\\
&&+
	\frac{\left(-\Theta\sin\psi+\frac{\Phi-\Psi\cos\theta}{\sin\theta}\cos\psi-\Psi_2^{(n-2)}(\upeta, \upxi)+\frac{\upxi_{n-1, 3}}{\upxi_{n-1, 2}}(\Psi-\Psi^{(n-2)}_3(\upeta, \upxi))\right)^2
	}{2\mu_n\upxi_{n, 3}^2}\nonumber\\
&&+
\frac{\left(
	-\Theta\cos\psi-\frac{\Phi-\Psi\cos\theta}{\sin\theta}\sin\psi+\Psi_1^{(n-1)}(\upeta, \upxi)
	\right)^2
}{2\mu_n\upxi_{n, 3}^2}	\nonumber \\
&& -\sum_{0\le i<j\le n}\frac{m_{i+1}m_{j+1}}{\left|-\frac{M_{i}}{M_{i+1}}\upxi_{i}+
		\sum_{k=i+1}^{j-1}\frac{m_{k+1}}{M_{k+1}}\upxi_k+
		\upxi_{j}
		\right|} .\nonumber \end{eqnarray}
As anticipated, $H$ is $\varphi$--independent. Moreover, when $\rm C$ vanishes, so 
\begin{eqnarray}\label{condition}
\Phi=\Theta=\Psi=0
\end{eqnarray}
then $H$ turns to depend only on $(\upeta, \upxi)$:
\begin{eqnarray}\label{H0OLD} H_0(\upeta, \upxi)&:=&H(0, 0, 0, \upeta, 
\theta, \psi, \upxi)\nonumber\\
&=&\sum_{j=1}^{n}\frac{|\upeta_j|^2}{2\mu_j}
	+
	\frac{(\Psi_{3}^{(n-2)})^2}{2\mu_{n-1}\upxi_{n-1, 2}^2}+
	\frac{\left(-\Psi_2^{(n-2)}(\upeta, \upxi)-\frac{\upxi_{n-1, 3}(\upeta, \upxi)}{\upxi_{n-1, 2}}\Psi^{(n-2)}_3(\upeta, \upxi)\right)^2
	}{2\mu_n\upxi_{n, 3}^2}\nonumber\\
&+&
\frac{\left(
	\Psi_1^{(n-1)}(\upeta, \upxi)
	\right)^2
}{2\mu_n\upxi_{n, 3}^2}	-\sum_{0\le i<j\le n}\frac{m_{i+1}m_{j+1}}{\left|-\frac{M_{i}}{M_{i+1}}\upxi_{i}+
		\sum_{k=i+1}^{j-1}\frac{m_{k+1}}{M_{k+1}}\upxi_k+
		\upxi_{j}
		\right|} \ \forall\ (\varphi, \theta, \psi)\,.\end{eqnarray}

\subsection{Proof of~\eqref{etas}}\label{subsec:Proof-of-etas}
We proceed by steps. 	During the computations, we shall denote ${\rm C}_j:=x_j\times y_j$, so that ${\rm C}=\sum_{j=1}^n{\rm C}_n$. 
\vskip.1in
\noindent
$\bullet$ The expression of $\rm C$  is found out decomposing this vector on the basis $\{e_3, \gamma, {f}_3\}$:
$${\rm C}=\alpha e_3+\beta \gamma+\delta {f}_3$$
Note that such basis is normal, but not orthogonal. Precisely, by definition, we have
$$e_3\perp \gamma\,,\qquad \gamma\perp {f}_3\,,\quad e_3\cdot {f}_3=\cos\theta\,.$$
Taking then the inner products of ${\rm C}$ with $e_3$, $\gamma$ and ${f}_3$, and using the definitions of $\Theta$, $\Phi$ and $\Psi$ in~\eqref{ThetaPhi}, we find the system
$$\left\{
\begin{array}{lll}
\alpha+\delta\cos\theta=\Phi\\
\beta=\Theta\\
\alpha\cos\theta+\delta=\Psi
\end{array}
\right.$$
which  solves as
$$\left\{
\begin{array}{lll}
\displaystyle\alpha=\frac{\Phi-\Psi\cos\theta}{\sin^2\theta}\\\\
\displaystyle\beta=\Theta\\\\
\displaystyle\delta=\frac{\Psi-\Phi\cos\theta}{\sin^2\theta}
\end{array}
\right.$$
Then we have
\begin{eqnarray}\label{C*****}
{\rm C}&=&\frac{\Phi-\Psi\cos\theta}{\sin^2\theta}e_3+\Theta \gamma+\frac{\Psi-\Phi\cos\theta}{\sin^2\theta}{f}_3\nonumber\\
&=&{\mathcal R}_3\left(\varphi\right) \mathcal{R}_1(\theta)\mathcal{R}_3(\psi) \left(
\frac{\Phi-\Psi\cos\theta}{\sin^2\theta}
\mathcal{R}_3(-\psi)\mathcal{R}_1(-\theta) e_3^*+\Theta \mathcal{R}_3(-\psi)e_1^*+
\frac{\Psi-\Phi\cos\theta}{\sin^2\theta}e_3^*
\right)\nonumber\\
&=&{\mathcal R}_3\left(\varphi\right) \mathcal{R}_1(\theta)\mathcal{R}_3(\psi) \left(\begin{array}{ccc}\displaystyle\Theta\cos\psi+\frac{\Phi-\Psi\cos\theta}{\sin\theta}\sin\psi\\\\
\displaystyle-\Theta\sin\psi+\frac{\Phi-\Psi\cos\theta}{\sin\theta}\cos\psi\\\\
\displaystyle\Psi
\end{array}
\right)\,.
\end{eqnarray}
where we used (\ref{eq:appR3R1e1g}) from Lemma~\ref{lem:EulerAngles}.
\vskip.1in
\noindent
$\bullet$ 
We check that the expression of $\upzeta_{n-1, 1}$ is given by
\begin{eqnarray}\label{etan-11}
\upzeta_{n-1, 1}&=&y_{n-1}\cdot {f}_1=y_{n-1}\cdot \frac{x_n\times x_{n-1}}{|x_n\times x_{n-1}|}=x_{n}\cdot \frac{x_{n-1}\times y_{n-1}}{|x_n\times x_{n-1}|}\nonumber\\
&=& \frac{x_n\cdot({\rm C}-{\rm C}_n-\sum_{j=1}^{n-2}{\rm C}_j)}{|x_n\times x_{n-1}|}= \frac{x_n\cdot({\rm C}-\sum_{j=1}^{n-2}{\rm C}_j)}{|x_n\times x_{n-1}|}\nonumber\\
&=&-\frac{1}{\upxi_{n-1, 2}}(f_3\cdot {\rm C}-f_3\cdot\sum_{j=1}^{n-2}{\rm C}_j)=-\frac{1}{\upxi_{n-1, 2}}(\Psi-\Psi^{(n-2)}_3)
\end{eqnarray}
We have
\begin{eqnarray*}
\upzeta_{n-1, 1}=y_{n-1}\cdot {f}_1=y_{n-1}\cdot \frac{x_n\times x_{n-1}}{|x_n\times x_{n-1}|}=x_n \cdot  \frac{x_{n-1}\times y_{n-1}}{|x_n\times x_{n-1}|}
\end{eqnarray*}
Since from (\ref{xn xn-1})  $x_n=|x_n| f_3=\upxi_{n,3} f_3$ and $|x_n \times x_{n-1}|=|\upxi_n \times \upxi_{n-1}|=|\upxi_{n,3} \upxi_{n-1,2}|= -\upxi_{n,3} \upxi_{n-1,2}$ {because
of signs in (\ref{eq:upxi-ranges})}. Continuing we obtain
\begin{eqnarray*}
\upzeta_{n-1, 1}=-\frac{1}{\upxi_{n-1,2}} f_3 \cdot \left({\rm C} - {\rm C}_n - \sum_{j=1}^{n-2}{\rm C}_j \right)= -\frac{1}{\upxi_{n-1,2}}  \left( f_3 \cdot {\rm C}  - f_3 \cdot {\rm C}_n - \sum_{j=1}^{n-2} f_3\cdot {\rm C}_j\right)
\end{eqnarray*} 
Observe that $f_3 \cdot {\rm C}=\Phi$, $f_3 \cdot {\rm C}_n = f_3 \cdot (x_n \times y_n)=0$ because $f_3$ and $x_n$ are collinear and for $1\le j \leq n-2$ it holds
\begin{eqnarray*}
  f_3 \cdot {\rm C}_j &=& (\mathcal{R}_3(\varphi) \mathcal{R}_1(\theta)\mathcal{R}_3(\psi) e^*_3) \cdot (\mathcal{R}_3(\varphi) \mathcal{R}_1(\theta)\mathcal{R}_3(\psi) (\upxi_j \times \upzeta_j))= e^*_3 \cdot (\upxi_j \times \upzeta_j)=e^*_3 \cdot (\upxi_j \times \upeta_j) \\
  &=& \upxi_{j,1}  \upeta_{j,2} - \upxi_{j,2}  \upeta_{j,1} 
\end{eqnarray*}
We then arrive at~\eqref{etan-11}.

\vskip.1in
\noindent
$\bullet$ Using~\eqref{etan-11} and~\eqref{xn xn-1}, we find the expression of ${\rm C}_{n-1}$ as
\begin{eqnarray}\label{Cn-1}
{\rm C}_{n-1}&=&{\mathcal R}_3(\varphi){\mathcal R}_1(\theta){\mathcal R}_3(\psi)
\left(
\begin{array}{lll}
\displaystyle\upxi_{n-1, 2}\upzeta_{n-1, 3}-\upxi_{n-1, 3}\upzeta_{n-1, 2}\\\\
\displaystyle\upxi_{n-1, 3}\upzeta_{n-1, 1}-\upxi_{n-1, 1}\upzeta_{n-1, 3}\\\\
\displaystyle\upxi_{n-1, 1}\upzeta_{n-1, 2}-\upxi_{n-1, 2}\upzeta_{n-1, 1}
\end{array}
\right)\nonumber\\\nonumber\\
&=&{\mathcal R}_3(\varphi){\mathcal R}_1(\theta){\mathcal R}_3(\psi)
\left(
\begin{array}{ccc}
\displaystyle\upxi_{n-1, 2}\upeta_{n-1, 3}-\upxi_{n-1, 3}\upeta_{n-1, 2}\\\\
\displaystyle-\frac{\upxi_{n-1, 3}}{\upxi_{n-1, 2}}(\Psi-\Psi^{(n-2)}_3)\\\\
\displaystyle\Psi-\Psi^{(n-2)}_3
\end{array}
\right)
\end{eqnarray}
$\bullet$ Using the definitions in~\eqref{Psis} and Equations~\eqref{C*****} and~\eqref{Cn-1}, we find the expression of ${\rm C}_{n}$, which is
\begin{eqnarray}\label{Cn}
{\rm C}_{n}&=&{\rm C}-{\rm C}_{n-1}-\sum_{j=1}^{n-2}{\rm C}_j={\mathcal R}_3(\varphi){\mathcal R}_1(\theta){\mathcal R}_3(\psi)\cdot\nonumber\\
&&
\left(\begin{array}{ccc}\displaystyle\Theta\cos\psi+\frac{\Phi-\Psi\cos\theta}{\sin\theta}\sin\psi-\Psi_1^{(n-1)}\\\\
\displaystyle-\Theta\sin\psi+\frac{\Phi-\Psi\cos\theta}{\sin\theta}\cos\psi-\Psi_2^{(n-2)}+\frac{\upxi_{n-1, 3}}{\upxi_{n-1, 2}}(\Psi-\Psi^{(n-2)}_3)\\\\
\displaystyle0
\end{array}
\right)
\end{eqnarray}
$\bullet$  We check the formulae for $\upzeta_{n, 1}$ and $\upzeta_{n, 2}$ in~\eqref{etas}.  We need to identify $\upzeta_{n, 1}$ and $\upzeta_{n, 2}$ through
\begin{eqnarray*}
  y_n = \upzeta_{n, 1} f_1 + \upzeta_{n, 2}f_2 + \upzeta_{n,3}f_3
\end{eqnarray*}
We have the splitting
\begin{eqnarray*}
  y_n = (y_n,f_3) f_3 + (f_3 \times y_n) \times f_3
\end{eqnarray*}
so we reduce to identify $\zeta_{n, 1}$, $\zeta_{n, 2}$ through
\begin{eqnarray}\label{phi mapNEW}(f_3 \times y_n) \times f_3=\zeta_{n, 1} f_1+\zeta_{n, 2} f_2\,.\end{eqnarray}
Let  $\tilde{\rm C}_n$ be the vector appearing in the last row of~\eqref{Cn}. Using~\eqref{Cn} we obtain (recall that $f_3=\frac{x_n}{|x_n|}$ and $|x_n|=\upxi_{n,3}$, $\mathcal{R}(\varphi,\theta,\psi)=(f_1,f_2,f_3)$)
\begin{eqnarray*}
  (f_3 \times y_n) \times f_3&=& \frac{1}{|x_n|} (x_n \times y_n) \times f_3= \frac{1}{\upxi_{n,3}} {\rm C}_n \times f_3\\
 & =& \frac{1}{\upxi_{n,3}} (\mathcal{R}(\varphi,\theta,\psi)\tilde{\rm C}_n) \times  (\mathcal{R}(\varphi,\theta,\psi)e_3^*) 
  = \frac{1}{\upxi_{n,3}} \mathcal{R}(\varphi,\theta,\psi)(\tilde{\rm C}_n \times e_3^*) \\
  &=& \frac{1}{\upxi_{n,3}}\mathcal{R}(\varphi,\theta,\psi)\left(\begin{array}{ccc}\displaystyle
  -\Theta\sin\psi+\frac{\Phi-\Psi\cos\theta}{\sin\theta}\cos\psi-\Psi_2^{(n-2)}+\frac{\upxi_{n-1, 3}}{\upxi_{n-1, 2}}(\Psi-\Psi^{(n-2)}_3)\\
\displaystyle- \Theta\cos\psi-\frac{\Phi-\Psi\cos\theta}{\sin\theta}\sin\psi+\Psi_1^{(n-1)}\\
0
\end{array}
\right) \\
&=& \frac{1}{\upxi_{n,3}}\left(-\Theta\sin\psi+\frac{\Phi-\Psi\cos\theta}{\sin\theta}\cos\psi-\Psi_2^{(n-2)}+\frac{\upxi_{n-1, 3}}{\upxi_{n-1, 2}}(\Psi-\Psi^{(n-2)}_3) \right)f_1\\
 &+& \frac{1}{\upxi_{n,3}}\left(- \Theta\cos\psi-\frac{\Phi-\Psi\cos\theta}{\sin\theta}\sin\psi+\Psi_1^{(n-1)} \right)f_2
\end{eqnarray*}
which allows to identify $\upzeta_{n, 1}$ and $\upzeta_{n, 2}$ as in~\eqref{etas}, via~\eqref{phi mapNEW}. $\qquad\square$

\section{Coordinates suited to total collision}
Leaving $(\Phi, \Theta, \Psi, \varphi, \theta, \psi)$ unvaried, we modify, canonically,  the coordinates $(\upeta, \upxi)$   in order to have a set better suited to total collision. We proceed along the ideas of~\cite{moeckelM2025}, adapted to our situation. In order to have lighter notations, we write \begin{eqnarray*}R:=\upeta_{n, 3}\,,\qquad r:=\xi_{n, 3}\,.
\end{eqnarray*}
Then we define
\begin{eqnarray*}
(\sigma, \varrho):=\left(\frac{\hat\upxi}{r }\,,\  |\upxi|_\mu\right)\end{eqnarray*}
where
$$\hat\upxi:=(\upxi_1,\ldots, \upxi_{n-1})\,,\qquad |\upxi|_\mu:=\sqrt{\sum_{i=1}^{n}\mu_i |\upxi_i|^2}\,.$$
Condition~\eqref{eq:upxi-ranges} implies
\begin{eqnarray}
\varrho>0\,,\qquad \sigma\in \Sigma:=\Big\{\sigma\in {\mathbb R}^{3(n-2)+2}:\  \upsigma_{n-1, 2}\in \mathbb R_-\setminus\{0\}\Big\}\,.  \label{eq:upsigma-ranges}
\end{eqnarray}
The inverse change is
\begin{eqnarray}
\label{inverse formulae}\upxi=(\hat\upxi\,, r)= \varrho\frac{(\sigma, 1)}{|(\sigma, e)|_\mu}\,.
\end{eqnarray}
The notation~\eqref{inverse formulae} stands for
$$(\sigma,1):=(\sigma_1,\dots,\sigma_{n-1},1) \in \mathbb{R}^{3(n-1)+1}\,,\qquad (\sigma,e) := (\sigma_1,\dots,\sigma_{n-1},(0,0,1))$$
We determine the momenta $({\mathcal S}, {\mathcal R})$   conjugated to $(\sigma, \varrho)$ 
through an exact and symplectic transformation. Namely,
 we want to obtain
\begin{equation}
  {\mathcal S}\cdot d\sigma+{\mathcal R}d\varrho=\sum_{i=1}^{n-1}\upeta_i \cdot d\upxi_j + R dr
\end{equation}
We have
\begin{eqnarray}\label{Mathieu}
{\mathcal S}\cdot d\sigma+{\mathcal R}d\varrho&=&{\mathcal S}\cdot \left(\frac{d\hat\upxi}{r }-\hat\upxi\frac{dr}{r ^2}\right)+{\mathcal R}\left(\frac{\sum_{i=1}^{n-1}\mu_i \upxi_i \cdot d\upxi_i+\mu_n r  dr }{ |\upxi|_\mu}\right)\nonumber\\
&=&{\mathcal S}\cdot \left(\frac{|(\sigma, e)|_\mu}{\varrho}d\upxi-\sigma|(\sigma, e)|_\mu\frac{dr}{\varrho}\right)\nonumber\\
&+&\frac{\mathcal R}{|(\sigma, e)|_\mu}\left({\sum_{i=1}^{n-1}\mu_i \sigma_i\cdot d\upxi_i+\mu_n  dr }\right)\end{eqnarray}
where we have used~\eqref{inverse formulae}.
Equation~\eqref{Mathieu} provides
\begin{eqnarray}\label{inverse formulaeNEW}\upeta_i&=&\frac{|(\sigma, e)|_\mu}{\varrho}{\mathcal S}_i+\frac{\mathcal R}{|(\sigma, e)|_\mu}\mu_i\sigma_i\,,\quad i=1\,,\ldots\,,n-1\nonumber\\
R &=&-\frac{|(\sigma, e)|_\mu}{\varrho}\sum_{i=1}^{n-1}{\mathcal S}_i\cdot\sigma_i +\mu_n \frac{\mathcal R}{|(\sigma, e)|_\mu}
\end{eqnarray}
When the change~\eqref{inverse formulae},~\eqref{inverse formulaeNEW} is inserted into the Hamiltonian~\eqref{HNEWNEW}, the functions $\Psi_k^{(p)}(\upeta, \upxi)$ thereby appearing are transformed into $\Psi_k^{(p)}({\mathcal S}, \sigma)$ because of their skew--symmetry (see the definition in~\eqref{Psis}). 
After a few computations, it can be seen that $H$ becomes
\begin{eqnarray}\label{Hcomplete} H&=&\frac{{\mathcal R}^2}{2}+\frac{|(\sigma, e)|^2_\mu}{2\varrho^2}\left[\sum_{i=1}^{n-1}\frac{|{\mathcal S}_i|^2}{\mu_i}+\frac{1}{\mu_n}\left(\sum_{i=1}^{n-1}{\mathcal S}_i\cdot\sigma_i\right)^2+
	\frac{(\Psi-\Psi_{3}^{(n-2)})^2}{2\mu_{n-1}\sigma_{n-1, 2}^2}\right.
\nonumber\\
&&+
\frac{\left(-\Theta\sin\psi+\frac{\Phi-\Psi\cos\theta}{\sin\theta}\cos\psi-\Psi_2^{(n-2)}+\frac{\sigma_{n-1, 3}}{\sigma_{n-1, 2}}(\Psi-\Psi^{(n-2)}_3)\right)^2
}{2\mu_n}\nonumber\\
&&+\left.
\frac{\left(
	-\Theta\cos\psi-\frac{\Phi-\Psi\cos\theta}{\sin\theta}\sin\psi+\Psi_1^{(n-1)}
	\right)^2
}{2\mu_n}\right]	\nonumber\\
&&-\frac{|(\sigma, e)|_\mu}{\varrho}\sum_{0\le i<j\le n}\frac{m_{i+1}m_{j+1}}{\left|-\frac{M_{i}}{M_{i+1}}\sigma_{i}+
	\sum_{k=i+1}^{j-1}\frac{m_{k+1}}{M_{k+1}}\sigma_k+
	\sigma_{j}
	\right|} \end{eqnarray}
	Instead, $H_0$ in~\eqref{H0OLD}  becomes
\begin{eqnarray}\label{H0}
H_0&=&\frac{{\mathcal R}^2}{2}+\frac{|(\sigma, e)|^2_\mu}{2\varrho^2}\left[\sum_{i=1}^{n-1}\frac{|{\mathcal S}_i|^2}{\mu_i}+\frac{1}{\mu_n}\left(\sum_{i=1}^{n-1}{\mathcal S}_i\cdot\sigma_i\right)^2+
\frac{\Psi_{3}^{(n-2)}({\mathcal S}, \sigma)^2}{2\mu_{n-1}\sigma_{n-1, 2}^2}\right.
\nonumber\\
&&+
\frac{\left(\Psi_2^{(n-2)}+\frac{\sigma_{n-1, 3}}{\sigma_{n-1, 2}}\Psi^{(n-2)}_3\right)^2
}{2\mu_n}+\left.
\frac{\left(
\Psi_1^{(n-1)}
	\right)^2
}{2\mu_n}\right]\nonumber\\
&&	-\frac{|(\sigma, e)|_\mu}{\varrho}\sum_{0\le i<j\le n}\frac{m_{i+1}m_{j+1}}{\left|-\frac{M_{i}}{M_{i+1}}\sigma_{i}+
	\sum_{k=i+1}^{j-1}\frac{m_{k+1}}{M_{k+1}}\sigma_k+
	\sigma_{j}
	\right|} 
\end{eqnarray}

\section{Regularizations of~\eqref{existence conditionsNEW}}\label{sec: regularization}


Condition~\eqref{existence conditionsNEW} strongly affects  the Hamiltonian vector--field  of $H$ -- especially for the part which governs the motion of\footnote{Computing the Hamilton equations of $H$ in~\eqref{Hcomplete} for $\varphi$, $\theta$, $\psi$ and next assuming~\eqref{condition}, one obtains
\begin{eqnarray*}
	\left\{
	\begin{array}{lll}
		\displaystyle\dot\varphi=-\frac{|(\sigma, e)|^2_\mu}{\mu_n\varrho^2}\left[
		\frac{\cos\psi}{\sin\theta}
		\left(\Psi_2^{(n-2)}({\mathcal S}, \sigma)+\frac{\sigma_{n-1, 3}}{\sigma_{n-1, 2}}\Psi^{(n-2)}_3({\mathcal S}, \sigma)\right)
		+\frac{\sin\psi}{\sin\theta}\Psi_1^{(n-1)}({\mathcal S}, \sigma)
		\right]\\\\
		\displaystyle\dot\theta=\frac{|(\sigma, e)|^2_\mu}{\mu_n\varrho^2}\left[
		\sin\psi
		\left(\Psi_2^{(n-2)}({\mathcal S}, \sigma)+\frac{\sigma_{n-1, 3}}{\sigma_{n-1, 2}}\Psi^{(n-2)}_3({\mathcal S}, \sigma)\right)
		-	\cos\psi
		\Psi_1^{(n-1)}({\mathcal S}, \sigma)
		\right]\\\\
		\displaystyle\dot\psi=
		\frac{|(\sigma, e)|^2_\mu}{\varrho^2}\left[
		-\frac{\Psi_3^{(n-2)}({\mathcal S}, \sigma)}{\mu_{n-1}\sigma_{n-1, 2}^2}+
		\frac{1}{\mu_n}\left(
		-\left(-\frac{\cos\theta}{\sin\theta}\cos\psi+\frac{\sigma_{n-1,3}}{\sigma_{n-1, 2}}\right)
		\right.\right.\\
		\qquad\left.\left.
		\left(\Psi_2^{(n-2)}({\mathcal S}, \sigma)+\frac{\sigma_{n-1, 3}}{\sigma_{n-1, 2}}\Psi^{(n-2)}_3({\mathcal S}, \sigma)\right)+\frac{\cos\theta}{\sin\theta}\sin\psi
		\Psi_1^{(n-1)}({\mathcal S}, \sigma)
		\right)\right]
	\end{array}
	\right.	
\end{eqnarray*}
which are singular when $\theta=0$, $\pi$.
}  $\varphi$, $\theta$ and $\psi$ --  showing  singularities in correspondence of the values $0$, $\pi$ for the coordinate $\theta$, even under conditions~\eqref{condition}. In order not to overload the result with unnecessary restrictions, in this section we switch to new coordinates which make the relevant vector--field regular also when~\eqref{existence conditionsNEW} is not verified. 
First, we  define  $w=(u, v, \alpha)\in {\mathbb R}^2\times{\mathbb T}$ via
\begin{eqnarray}\label{regularization}
\left\{
\begin{array}{lll}
\displaystyle u=
\sin\theta\cos\psi\\
\displaystyle v=
\sin\theta\sin\psi\\
\displaystyle \alpha=\varphi+\psi
\end{array}
\right.
\end{eqnarray}
We show in a moment that relation~\eqref{regularization}  triggers  two canonical transformations  between the 6--ples $(\Phi, \Theta, \Psi, \varphi, \theta, \psi)$ and $(W, w)=(U, V, A, u, v, \alpha)$ (hence, keeping $({\mathcal R}, {\mathcal S}, \varrho, \sigma)$ unvaried)
{\small
\begin{eqnarray}\label{regularizations}\phi_-^{\rm reg}:\ (U, V, A, u, v, \alpha)\in{\mathbb R}^3\times B^2_1(0,0)\setminus\{(0,0)\}\times{\mathbb T}\to(\Phi, \Theta, \Psi, \varphi, \theta, \psi)\in {\mathbb R}^3\times {\mathbb T}\times\left(0, \frac{\pi}{2}\right) \times{\mathbb T}
\nonumber\\
\phi_+^{\rm reg}:\ (U, V, A, u, v, \alpha)\in{\mathbb R}^3\times B^2_1(0,0)\setminus\{(0,0)\}\times{\mathbb T}\to(\Phi, \Theta, \Psi, \varphi, \theta, \psi)\in {\mathbb R}^3\times {\mathbb T}\times\left(\frac{\pi}{2}, \pi\right) \times{\mathbb T}
\end{eqnarray}}
   defined through
\begin{eqnarray}\label{symplectic regularization}
\left\{
\begin{array}{ll}
\Phi=A\\\\
\displaystyle\Theta=Uu\frac{\sqrt{1-u^2-v^2}}{\sqrt{u^2+v^2}}+Vv \frac{\sqrt{1-u^2-v^2}}{\sqrt{u^2+v^2}}\\\\
\Psi=-Uv+Vu+A
\end{array}
\right.\quad \left\{\begin{array}{lll}
\displaystyle\varphi=\alpha-\arg(u, v)\\\\
\displaystyle \theta=\frac{\pi}{2}\mp\left(\frac{\pi}{2}-\sin^{-1}\sqrt{u^2+v^2}\right)\\\\
\displaystyle  \psi=\arg(u, v)
  \end{array}
  \right.
\end{eqnarray}

where 
$\arg(u, v)\in \mathbb T$ is the unique angle defined through
$$\cos\arg(u, v)=\frac{u}{\sqrt{u^2+v^2}}\,,\qquad \sin\arg(u, v)=\frac{v}{\sqrt{u^2+v^2}}\qquad \forall\ (u, v)\in B^2_1(0,0)\setminus\{(0,0)\}\,.$$


\begin{remark}\label{zero ang mom}\rm
Under the constraints specified in~\eqref{regularizations}, the transformations $\phi^{\rm reg}_\pm$
are invertible. In particular, the
zero angular momentum condition~\eqref{condition} remains unvaried:
\begin{eqnarray}\label{conditionNEW}U=V=A=0\,.\end{eqnarray}
\end{remark}

\begin{remark}\rm 
In Proposition~\ref{prop:reg} below we shall prove that, under the constraint~\eqref{conditionNEW}, the triples $w=(u, v, \alpha)$, pre--images of $\phi^{\rm reg}_+$ or $\phi^{\rm reg}_+$, satisfy the same ODE, which here we denote as
\begin{eqnarray}\label{sameODE}\dot w=f({\mathcal S}, \sigma, u, v)\,.\end{eqnarray}The main point is that, for fixed $({\mathcal S}, \sigma)$, $f({\mathcal S}, \sigma, u, v)$ is $C^1$ in the open ball $B^2_1(0,0)$ and $C^0 $ in its clusure $\overline B^2_1(0,0)$, hence the  new coordinates are allowed to visit zones in phase space which were prevented for the previous ones. In particular, the value $(u, v)=(0,0)$ corresponds to $\theta=0$ for $\phi^{\rm reg}_+$, to $\theta=\pi$ for $\phi^{\rm reg}_-$.
Of course, the motion of $w=(u, v, \alpha)$, solution of~\eqref{sameODE}, does not allow to infer the one of the Euler triple $(\varphi, \theta, \psi)$, because of the indeterminacy of the sign in~\eqref{symplectic regularization}. However, it allows to prove that
 $(\varphi(t), \theta(t), \psi(t))$ has a finite limit as $t\to T^-$, just proving that this holds for $w(t)=(u(t), v(t), \alpha(t))$.
\end{remark}
First of all, we aim to prove that
\begin{proposition} The changes $\phi^{\rm reg}_\pm$ in~\eqref{regularizations}$\div$\eqref{symplectic regularization}
are exact symplectic.  
\end{proposition}
\proof Rewriting the former system (using the latter) as
\begin{eqnarray*}
\left\{
\begin{array}{ll}
\Phi=A\\
\Theta=U\cos\theta\cos\psi+V\cos\theta\sin\psi\\
\Psi=-U\sin\theta \sin\psi+V\sin\theta \cos\psi+A
\end{array}
\right.
\end{eqnarray*}
we have that the standard $1$--form is conserved, for both $\phi^{\rm reg}_\pm$:
\begin{eqnarray*}
U du+Vdv+Ad\alpha&=&U(\cos\theta\cos\psi d\theta-\sin\theta \sin\psi d\psi)+V(\cos\theta \sin\psi d\theta+\sin\theta \cos\psi d\psi)\nonumber\\
&+&A(d\varphi+d\psi)
\nonumber\\
&=&\Phi d\varphi+\Theta d\theta+\Psi d\psi\,.\qquad \square
\end{eqnarray*}

\begin{proposition}\label{prop:reg}
	The motion of the  pre--image $w=(u, v, \alpha)$ of $\phi^{\rm reg}_+$ $($respectively, of $\phi^{\rm reg}_-$$)$ when the total angular momentum $\rm C$ vanishes is ruled by the ODE
	\begin{eqnarray}\label{uva motion}
	\left\{
	\begin{array}{lll}
\dot u=\frac{|(\sigma, e)|^2_\mu}{2\varrho^2}\left[\frac{v\Psi^{(n-2)}_3({\mathcal S}, \sigma)}{\mu_{n-1}\sigma_{n-1, 2}^2}+\frac{\Psi_2^{(n-2)}({\mathcal S}, \sigma)
	+\frac{\sigma_{n-1, 3}}{\sigma_{n-1, 2}}\Psi_3^{(n-2)}({\mathcal S}, \sigma)
	}{\mu_n}\frac{\sigma_{n-1, 3}}{\sigma_{n-1, 2}} v-\frac{\Psi_1^{(n-1)}({\mathcal S}, \sigma)}{\mu_n}\sqrt{1-u^2-v^2}\right]\\\\
\dot v=\frac{|(\sigma, e)|^2_\mu}{2\varrho^2}\left[-\frac{u\Psi_3^{(n-2)}({\mathcal S}, \sigma)}{\mu_{n-1}\sigma_{n-1, 2}^2}-
	\frac{\Psi_2^{(n-2)}({\mathcal S}, \sigma)
	+\frac{\sigma_{n-1, 3}}{\sigma_{n-1, 2}}\Psi_3^{(n-2)}({\mathcal S}, \sigma)
	}{\mu_n}\left(\frac{\sigma_{n-1, 3}}{\sigma_{n-1, 2}} u
	-\sqrt{1-u^2-v^2}
	\right)\right]\\\\
\dot\alpha=\frac{|(\sigma, e)|^2_\mu}{2\varrho^2}\left[-\frac{\Psi^{(n-2)}_3({\mathcal S}, \sigma)}{\mu_{n-1}\sigma_{n-1, 2}^2}-\frac{\Psi_2^{(n-2)}({\mathcal S}, \sigma)
	+\frac{\sigma_{n-1, 3}}{\sigma_{n-1, 2}}\Psi_3^{(n-2)}({\mathcal S}, \sigma)}{\mu_n}\left(
\frac{u}{1+\sqrt{1-u^2-v^2}}+\frac{\sigma_{n-1, 3}}{\sigma_{n-1, 2}} 	\right)\right]
	\end{array}
	\right.
	\end{eqnarray}
\end{proposition}

\proof 
We compute $H$ after the change
\eqref{symplectic regularization}, by replacing such formulae into the Hamiltonian~\eqref{Hcomplete}. We first compute how
the quantities
\begin{eqnarray*}
-\Theta\sin\psi+\frac{\Phi-\Psi\cos\theta}{\sin\theta}\cos\psi\,,\qquad -\Theta\cos\psi-\frac{\Phi-\Psi\cos\theta}{\sin\theta}\sin\psi
\end{eqnarray*}
appearing in~\eqref{Hcomplete} are transformed. We have

\begin{eqnarray*}
&&-\Theta\sin\psi+\frac{\Phi-\Psi\cos\theta}{\sin\theta}\cos\psi
\nonumber\\
&=&-\frac{\sqrt{1-u^2-v^2}}{\sqrt{u^2+v^2}}\left(Uu+Vv\right) \frac{v}{\sqrt{u^2+v^2}}+\frac{
A(1-\sqrt{1-u^2-v^2})
-(-Uv+Vu)
\sqrt{1-u^2-v^2}
}
{\sqrt{u^2+v^2}
}\frac{u}{\sqrt{u^2+v^2}}\nonumber\\
&=&-V\sqrt{1-u^2-v^2}+\frac{Au}{1+\sqrt{1-u^2-v^2}}
\end{eqnarray*}
and, similarly,
\begin{eqnarray*}
&&
-\Theta\cos\psi-\frac{\Phi-\Psi\cos\theta}{\sin\theta}\sin\psi
=-U\sqrt{1-u^2-v^2}-\frac{Av}{1+\sqrt{1-u^2-v^2}}\,.
\end{eqnarray*}
Then the expression of $H$ after
the composition with 
$\phi^{\rm reg}_+$ $($respectively, $\phi^{\rm reg}_-$$)$  is

\begin{eqnarray*} H&=&\frac{{\mathcal R}^2}{2}+\frac{|(\sigma, e)|^2_\mu}{2\varrho^2}\left[\sum_{i=1}^{n-1}\frac{|{\mathcal S}_i|^2}{\mu_i}+\frac{1}{\mu_n}\left(\sum_{i=1}^{n-1}{\mathcal S}_i\cdot\sigma_i\right)^2+
	\frac{(-Uv+Vu+A-\Psi_{3}^{(n-2)})^2}{2\mu_{n-1}\sigma_{n-1, 2}^2}\right.
\nonumber\\
&&+
\frac{\left(-V\sqrt{1-u^2-v^2}+\frac{Au}{1+\sqrt{1-u^2-v^2}}-\Psi_2^{(n-2)}+\frac{\sigma_{n-1, 3}}{\sigma_{n-1, 2}}(-Uv+Vu+A-\Psi^{(n-2)}_3)\right)^2
}{2\mu_n}\nonumber\\
&&+\left.
\frac{\left(
	-U\sqrt{1-u^2-v^2}-\frac{Av}{1+\sqrt{1-u^2-v^2}}+\Psi_1^{(n-1)}
	\right)^2
}{2\mu_n}\right]	\nonumber\\
&&-\frac{|(\sigma, e)|_\mu}{\varrho}\sum_{0\le i<j\le n}\frac{m_{i+1}m_{j+1}}{\left|-\frac{M_{i}}{M_{i+1}}\sigma_{i}+
	\sum_{k=i+1}^{j-1}\frac{m_{k+1}}{M_{k+1}}\sigma_k+
	\sigma_{j}
	\right|} \end{eqnarray*}
By Remark~\ref{zero ang mom}, under conditions~\eqref{regularizations}, the ODE for the triple $w=(u, v, \alpha)$ when the angular momentum vanishes is
$$\dot u=\left(\partial_U H\right)|_{\eqref{conditionNEW}}\,,\qquad \dot v=\left(\partial_V H\right)|_{\eqref{conditionNEW}}\,,\qquad \dot\alpha=\left(\partial_A H\right)|_{\eqref{conditionNEW}}$$
and corresponds to~\eqref{uva motion}. Even though~\eqref{uva motion} have been obtained under the constraint~\eqref{regularizations}, they hold for all $(u, v)\in \overline B^2(0,0)$ because the right hand side keeps to be continuous on all such domain. $\qquad \square$

\section{McGehee Blow--Up}

The function $H_0$ in~\eqref{H0} will be shortly denoted as
\begin{eqnarray}\label{H0simple}
{\rm H}_0({\mathcal R}, \varrho, {\mathcal S}, \sigma)=\frac{\mathcal R^2}{2}+\frac{T({\mathcal S}, \sigma)}{\varrho^2}-\frac{V(\sigma)}{\varrho}\,.
\end{eqnarray}
Recalling the definition of $\Sigma$ in~\eqref{eq:upsigma-ranges}, as $\frac{\mathcal R^2}{2}+\frac{T({\mathcal S}, \sigma)}{\varrho^2}$ is the kinetic part of the energy, then we have
\begin{lemma}\label{lem:T}
$T(\widetilde{\mathcal{S}},\sigma)$ is a positive definite quadratic form with respect to $\widetilde{\mathcal{S}}$, for all $\sigma\in \Sigma$.
\end{lemma}
We rescale momenta and time accordingly to
$${\cal R}=\frac{\widetilde{\cal R}}{\sqrt{\varrho}}\,,\qquad {\cal S}=\widetilde{\cal S}\sqrt{\varrho}\,,\qquad \varrho^{3/2}\frac{d}{dt}=\frac{d}{d\tau}\,.$$
The motion equations become
\begin{eqnarray}\label{equations}
\left\{
\begin{array}{lll}
\displaystyle\varrho'=\varrho\widetilde{\mathcal R}
\\\\
\displaystyle\widetilde{\mathcal R}'=\frac{\widetilde{\mathcal R}^2}{2}+2{T(\widetilde{\mathcal S}, \sigma)}-{V(\sigma)}\\\\
\displaystyle\widetilde{\mathcal S}'=-\partial_{\sigma}\left(
{T(\widetilde{\mathcal S}, \sigma)}-{V(\sigma)}
\right)-\frac{\widetilde{\mathcal R}\widetilde{\mathcal S}}{2}\\\\
\displaystyle\sigma'=\partial_{\widetilde{\mathcal S}}T(\widetilde{\mathcal S}, \sigma)
\end{array}
\right.
\end{eqnarray}
From~\eqref{H0simple}, and since $T$ is homogeneous of degree 2 with respect to $\mathcal S$, the energy of the system is
\begin{eqnarray}\label{total energy}h={\rm H}_0\left(\frac{\widetilde{\displaystyle\mathcal R}}{\sqrt{\varrho}}, \varrho, \widetilde{\mathcal S}\sqrt{\varrho}, \sigma\right)
=
\frac{\displaystyle\frac{\widetilde{\mathcal R}^2}{2}+{T(\widetilde{\mathcal S}, \sigma)}-{V(\sigma)}}{{\varrho}}
\end{eqnarray}
 Equations~\eqref{uva motion} become
\begin{eqnarray}\label{uva motionNEW}
	\left\{
	\begin{array}{lll}
 u'=\frac{|(\sigma, e)|^2_\mu}{2}\left[\frac{v\widetilde\Psi^{(n-2)}_3(\widetilde{\mathcal S}, \sigma)}{\mu_{n-1}\sigma_{n-1, 2}^2}+\frac{\widetilde\Psi_2^{(n-2)}(\widetilde{\mathcal S}, \sigma)
	+\frac{\sigma_{n-1, 3}}{\sigma_{n-1, 2}}\widetilde\Psi_3^{(n-2)}(\widetilde{\mathcal S}, \sigma)
	}{\mu_n}\frac{\sigma_{n-1, 3}}{\sigma_{n-1, 2}} v-\frac{\widetilde\Psi_1^{(n-1)}(\widetilde{\mathcal S}, \sigma)}{\mu_n}\sqrt{1-u^2-v^2}\right]\\\\
 v'=\frac{|(\sigma, e)|^2_\mu}{2}\left[-\frac{u\widetilde\Psi^{(n-2)}_3(\widetilde{\mathcal S}, \sigma)}{\mu_{n-1}\sigma_{n-1, 2}^2}-
	\frac{\widetilde\Psi_2^{(n-2)}(\widetilde{\mathcal S}, \sigma)
	+\frac{\sigma_{n-1, 3}}{\sigma_{n-1, 2}}\widetilde\Psi_3^{(n-2)}(\widetilde{\mathcal S}, \sigma)
	}{\mu_n}\left(\frac{\sigma_{n-1, 3}}{\sigma_{n-1, 2}} u
	-\sqrt{1-u^2-v^2}
	\right)\right]\\\\
\alpha'=\frac{|(\sigma, e)|^2_\mu}{2}\left[-\frac{\widetilde\Psi^{(n-2)}_3(\widetilde{\mathcal S}, \sigma)}{\mu_{n-1}\sigma_{n-1, 2}^2}-\frac{\widetilde\Psi_2^{(n-2)}(\widetilde{\mathcal S}, \sigma)
	+\frac{\sigma_{n-1, 3}}{\sigma_{n-1, 2}}\widetilde\Psi_3^{(n-2)}(\widetilde{\mathcal S}, \sigma)}{\mu_n}\left(
\frac{u}{1+\sqrt{1-u^2-v^2}}+\frac{\sigma_{n-1, 3}}{\sigma_{n-1, 2}} 	\right)\right]\,.
	\end{array}
	\right.
	\end{eqnarray}

\section{Main result}\label{sec:no-inifinite-spin}

By ``collisional manifold''  we mean the sub--manifold ${\mathcal C}=\{\varrho=0\}$. The system~\eqref{equations} shows that $\mathcal C$ is invariant. We look at the equilibria of the system~\eqref{equations}  on $\mathcal C$ which are limit point of solutions. 

\vskip.1in
\noindent
The main result of this paper is
\begin{theorem}\label{main} For all orbits of the system~\eqref{equations} approaching an isolated equilibrium on $\mathcal C$  such in a way that
\begin{eqnarray}\label{non collinearityNEW}
\sup_{\tau\in[T, +\infty)}\left|\frac{1}{\sigma_{n-1, 2}(\tau)}\right|<+\infty\,,\qquad \sup_{\tau\in[T, +\infty)}\left|\frac{\sigma_{n-1, 3}(\tau)}{\sigma_{n-1, 2}(\tau)}\right|<+\infty
\end{eqnarray}
for some $T>0$,
 the triple $w(\tau)=(u(\tau), v(\tau), \alpha(\tau))$ solving
\eqref{uva motionNEW}
has a finite limit as $\tau\to +\infty$.
\end{theorem}
Observe that conditions~\eqref{non collinearityNEW} correspond to~\eqref{non collinearity}, due to the following equalities
\begin{eqnarray*}
\frac{\sigma_{n-1, 3}(\tau)}{\sigma_{n-1, 2}(\tau)}=\frac{\upxi_{n-1, 3}(t)}{\upxi_{n-1, 2}(t)}=\frac{x_{n-1}(t)\cdot f_3(t)}{x_{n-1}(t)\cdot f_2(t)}=
\frac{x_{n-1}(t)\cdot \frac{x_n(t)}{|x_n(t)|}}{x_{n-1}(t)\cdot \frac{x_n(t)}{|x_n(t)|}\times \frac{x_n(t)\times x_{n-1}(t)}{|x_n(t)\times x_{n-1}(t)|}}=-\frac{x_n(t)\cdot x_{n-1}(t)}{|x_n(t)\times x_{n-1}(t)|}
\end{eqnarray*}
and, similarly,
\begin{eqnarray*}
\frac{1}{\sigma_{n-1, 2}(\tau)}=\frac{\upxi_{n, 3}(t)}{\upxi_{n-1, 2}(t)}=\frac{x_{n}(t)\cdot f_3(t)}{x_{n-1}(t)\cdot f_2(t)}=
\frac{x_{n}(t)\cdot \frac{x_n(t)}{|x_n(t)|}}{x_{n-1}(t)\cdot \frac{x_n(t)}{|x_n(t)|}\times \frac{x_n(t)\times x_{n-1}(t)}{|x_n(t)\times x_{n-1}(t)|}}=-\frac{|x_n(t)|^2}{|x_n(t)\times x_{n-1}(t)|}
\end{eqnarray*}
The proof of Theorem~\ref{main} follows from
the following lemma, which we shall prove later on.
\begin{lemma}\label{finite Sint}
For all orbits of the system~\eqref{equations} approaching an isolated equilibrium on $\mathcal C$, one has
\begin{eqnarray}\label{eq: finite Sint}\int_0^{+\infty}|\widetilde{\mathcal S}(\tau)|<+\infty\,.
\end{eqnarray}\end{lemma}
Here we prove how Theorem~\ref{main} follows from Lemma~\ref{finite Sint}.
\proof  {\bf of Theorem~\ref{main}} We write the triple $w=(u, v, \alpha)$ as
\begin{eqnarray*}
w(\tau)=w(0)+\int_0^\tau w'(s)ds\,,
\end{eqnarray*}
so
\begin{eqnarray*}
\lim_{\tau\to+\infty}w(\tau)=w(0)+\int_0^{+\infty} w'(\tau)d\tau\,,
\end{eqnarray*}
and we only need to check that the integrals at right hand side converge.
As the functions $\Psi^{(k)}_j(\widetilde{\mathcal S}, \sigma)$  at right hand side of 
\eqref{uva motionNEW} are linear with respect to $\widetilde{\mathcal S}$, we have a bound
\begin{eqnarray}\label{use non collinearity}
\max\left\{ |u'(\tau)|\,,\  |v'(\tau)|\,,\  |\alpha'(\tau)\right\}\le K|\widetilde{\mathcal S}(\tau))|
\end{eqnarray}
for some $K>0$, which
implies the convergence of the target integrals, by Lemma~\ref{finite Sint}.  
Note that bound in~\eqref{use non collinearity} uses  conditions in~\eqref{non collinearityNEW}, in order to have a control on the ratios $\frac{\sigma_{n-1, 3}(t)}{\sigma_{n-1, 2}(t)}$, $\frac{1}{\sigma_{n-1, 2}(t)}$ appearing  at tight hand side of~\eqref{uva motionNEW}.

$\qquad \square$

\section{Proof of Lemma~\ref{finite Sint}}
The proof of Lemma~\ref{finite Sint} is obtained adapting the ideas and techniques of~\cite{moeckelM2025} to our Hamiltonian analysis. This is what we discuss here.
\vskip.1in
\noindent

\begin{proposition}
The equilibria of the system~\eqref{equations} on $\mathcal C$
are given by  $(\varrho, \widetilde{\mathcal R}, \widetilde{\mathcal S}, \sigma)=(0, \widetilde{\mathcal R}_\star, 0, \sigma_\star)$
where 
\begin{eqnarray}\label{equilibrium}
\sigma_\star:\ \partial_\sigma V(\sigma_\star)=0\,,\qquad
\widetilde{\mathcal R}_\star=-\sqrt{V(\sigma_\star)}\,.
\end{eqnarray}
\end{proposition}
\proof
Let $(\varrho_\star, \widetilde{\mathcal R}_\star, \widetilde{\mathcal S}_\star, \sigma_\star)$ be an equilibrium on $\mathcal C$. As  $\varrho(\tau)\to \varrho_\star=0$ as $\tau\to+\infty$, then $\widetilde{\mathcal R}_\star<0$.
From the conservation of the energy~\eqref{total energy} and the equation for $\mathcal R'$ in~\eqref{equations}, we have
$$\left\{\begin{array}{lll}
\displaystyle\frac{\widetilde{\mathcal R}_\star^2}{2}+{T(\widetilde{\mathcal S}_\star, \sigma_\star)}-{V(\sigma_\star)}=0\\\\
\displaystyle\frac{\widetilde{\mathcal R}_\star^2}{2}+2{T(\widetilde{\mathcal S}_\star, \sigma_\star)}-{V(\sigma_\star)}=0
\end{array}
\right.$$
which gives, as $\widetilde{\mathcal R}_\star<0$,
$$\widetilde{\mathcal R}_\star=-\sqrt{V(\sigma_\star)}\,,\qquad T(\widetilde{\mathcal S}_\star, \sigma_\star)=0\,.$$
As $T(\widetilde{\mathcal S}_\star, \sigma_\star)$ is a positive definite quadratic form with respect to $\widetilde{\mathcal S}_\star$ for all $\sigma_\star\in \Sigma$ (Lemma~\ref{lem:T}), we have $\widetilde{\mathcal S}_\star=0$. $\qquad \square$
 
 \vskip.1in
 \noindent
We introduce the displacements coordinates from the equilibrium
\begin{eqnarray}\label{shift}\widehat{\mathcal R}=\widetilde{\mathcal R}-\widetilde{\mathcal R}_\star\,,\qquad \widehat\varrho=\varrho\,,\qquad \widehat{\mathcal S}=\widetilde{\mathcal S}\,,\qquad \widehat\sigma=\sigma-\sigma_\star\end{eqnarray}
and consider the linearized system, given by
\begin{eqnarray}\label{linearized system}\left\{
\begin{array}{lcr}
\widehat\varrho'=\widetilde{\mathcal R}_\star \widehat\varrho\\\\
\widehat{\mathcal R}'=\widetilde{\mathcal R}_\star \widehat{\mathcal R}\\\\
\widehat{\mathcal S}'=-\frac{\widetilde{\mathcal R}_\star }{2}\widehat{\mathcal S}+B\widehat\sigma\\\\
\widehat\sigma'=A \widehat{\mathcal S}
\end{array}
\right.
\end{eqnarray}
where
\begin{eqnarray}\label{AB}A:=\partial^2_{\widetilde{\mathcal S}} T(0, \sigma_\star)\,,\qquad B:=\partial^2_\sigma V(\sigma_\star)\end{eqnarray}
In deriving equations~\eqref{linearized system}  we have used
$$\partial_{\sigma}V(\sigma_\star)=\partial_{\widetilde{\mathcal S}} T(0, \sigma_\star)=\partial_{\sigma} T(0, \sigma_\star)=\partial^2_{\widetilde{\mathcal S}\sigma} T(0, \sigma_\star)=\partial^2_{\sigma} T(0, \sigma_\star)=0\,.$$
The matrices $A$, $B$ are symmetric, with $A$  positively definite { as the consequence of positive definiteness of $T$}. Let $\alpha$ be the unique symmetric matrix { positive definite} such that $\alpha^2=A$ and let $C\in \rm SO(3n-3)$ be such that 
\begin{eqnarray}\label{B}D:=C^{-1}\alpha B \alpha C
\end{eqnarray} 
is diagonal. Note that $\alpha B \alpha$ is Hermitian, so $C$ is  well  defined 
Observe that the matrix $B$ is singular if and only if $D$ has some zero entry on its principal diagonal.
We further change
\begin{eqnarray}\label{changeNEW}
\widehat{\mathcal S}=\alpha^{-1}Cw\,,\qquad 
\widehat\sigma=\alpha C s\,.\end{eqnarray}
\begin{lemma}\label{diagonalization}
	If the matrix $B$ is non--singular, $0$ is an hyperbolic equilibrium for the system~\eqref{linearized system}. If $B$ is singular,
	there is a non--trivial center space associated to it, described by
\begin{eqnarray}\label{center space}{\mathbb E}^{\rm c}=\big\{
	(\widehat{\mathcal R}, \widehat\varrho, w, s)\in \{0\}\times\{0\}\times\{0\}\times Ker(D)
	\}\,.\end{eqnarray}
	\end{lemma}
\begin{remark}\rm
	Turning back to the coordinates $(\widehat{\mathcal R},\widehat\varrho, \widehat{\mathcal S}, \widehat\sigma)$ via~\eqref{changeNEW} and~\eqref{B}, the center space is 
	\begin{eqnarray}\label{center space OLD coordinates}{\mathbb E}^{\rm c}=\big\{
	(\widehat{\mathcal R}, \widehat\varrho, \widehat{\mathcal S}, \widehat\sigma)\in \{0\}\times\{0\}\times\{0\}\times Ker(B)
	\}\,.\end{eqnarray}
	\end{remark}
\proof 
The change~\eqref{changeNEW} transforms the system~\eqref{linearized system}
into
\begin{eqnarray}\label{linearized systemNEW}\left\{
\begin{array}{lcr}
\widehat\varrho'=\widetilde{\mathcal R}_\star \widehat\varrho\\\\
\widehat{\mathcal R}'=\widetilde{\mathcal R}_\star \widehat{\mathcal R}\\\\
w'=-\frac{\widetilde{\mathcal R}_\star }{2}w+D s\\\\
s'=w
\end{array}
\right.
\end{eqnarray}
The two latter equations give

$$w_j''+\frac{\widetilde{\mathcal R}_\star }{2}w_j'-\frac{c_j}{2}w_j=0\qquad j=1\,,\ldots\,,3n-3\,.$$
where $c_j$ is the $j^{\rm th}$ entry on the diagonal of $D$.
 The solutions are
\begin{eqnarray}\label{solution}w_j(t)=\left\{
\begin{array}{lll}w_j^+ e^{\lambda^+_j t}+w_j^- e^{\lambda^-_j t}\quad &{\rm if}\ 
{\widetilde{\mathcal R}_\star ^2}+8 c_j\ne0\\
w_j^+ e^{\lambda^+_j t}+w_j^- t e^{\lambda^-_j t} &{\rm otherwise}
\end{array}
\right.\end{eqnarray}
where $w_j^\pm\in \mathbb C$ and $\lambda_j^\pm=-\frac{\widetilde{\mathcal R}_\star }{4}\pm\frac{1}{4}\sqrt{{\widetilde{\mathcal R}_\star ^2}+8 c_j}$. 
If $B$ is nonsingular, then $c_j\ne 0$ for all $j$. In such case, the real parts of $\lambda_j^\pm$ are all non--vanishing, both in the case $\lambda_j^\pm\in \mathbb R$, and in the case $\lambda_j^\pm\in \mathbb C\setminus\mathbb R$. So, if $B$ is nonsingular, the solution~\eqref{solution} is the composition of functions tending to or escaping from the equilibrium at an exponential rate. By the variated constants, it follows that whole equilibrium of the system~\eqref{linearized systemNEW} is hyperbolic, whence so is for the system~\eqref{linearized system}. Assume now that $B$ is singular. Using the coordinates $(\widehat{\mathcal R},  \widehat\varrho,  u, s)$, the center space ${\mathbb E}^{\rm c}$
is the product of the null space for $(\widehat{\mathcal R},  \widehat\varrho)$ and
the eigenspace associated to
 the zero eigenvalue of the matrix defining the two latter equations in
\eqref{linearized systemNEW} for $(u, w)$. In turn, the latter is defined by $w=0$ and $s:\ Ds=0$.
 $\qquad\square$
 \vskip.1in
\noindent
{\bf On the proof of Lemma~\ref{finite Sint}} We divide the equilibria of the system~\eqref{equations} on $\mathcal C$ into two classes. In the first class, we put the ones whose matrix $B$ is nonsingular. We have seen in the previous section that for such equilibria the center space is trivial, hence  so is the center manifold. Hence any solution approaching to the equilibria in this class  belongs to the stable manifold. Then $\widetilde{\mathcal S}(\tau)$ approaches $0$ at an exponential rate, and~\eqref{eq: finite Sint}  is verified. For the equilibria in the second class, we have to study the flow on the center manifold.

\subsection*{Flow on the center manifold}
We now consider an equilibrium of the system~\eqref{equations} to which corresponds a singular matrix $B$ in~\eqref{AB}. Let $k$ be the dimension of the associated center space $\mathbb E^{\rm c}$ in~\eqref{center space}, coinciding with the dimension of the eigenspace of $B$ relatively to the zero eigenvalue. To study the motions of the  system~\eqref{equations} on the center manifold, we observe that it decouples in two parts. The main part concerns the autonomous motion of the coordinates $(\widetilde{\cal R}, \widetilde{\cal S}, \sigma) $, which solve
\begin{eqnarray}\label{systemNEW}
\left\{
\begin{array}{lll}\displaystyle\widetilde{\mathcal R}'=\frac{\widetilde{\mathcal R}^2}{2}+2{T(\widetilde{\mathcal S}, \sigma)}-{V(\sigma)}\\\\
\displaystyle\widetilde{\mathcal S}'=-\partial_{\sigma}\left(
{T(\widetilde{\mathcal S}, \sigma)}-{V(\sigma)}
\right)-\frac{\widetilde{\mathcal R}\widetilde{\mathcal S}}{2}\\\\
\displaystyle\sigma'=\partial_{\widetilde{\mathcal S}}T(\widetilde{\mathcal S}, \sigma)
\end{array}
\right.
\end{eqnarray}
The solution for $\varrho$ is then given by
\begin{eqnarray}\label{varrho(tau)} \varrho(\tau)=\varrho(0)e^{\int_0^\tau\widetilde{\cal R}(s)ds}
\,.
\end{eqnarray}	
The following proposition will be useful for a further dimension reduction of the system~\eqref{systemNEW}.

\begin{proposition}\label{flow on center manifold}
	There exists a $k$--dimensional manifold ${\mathcal M}^{\rm c}$, tangent to the center space $\mathbb E^{\rm c}$ in~\eqref{center space OLD coordinates} at the equilibrium $(\widetilde{\mathcal R}, \varrho, \widetilde{\mathcal S}, \sigma)=(\widetilde{\mathcal R}_\star, 0, 0, \sigma_\star)$ defined in~\eqref{equilibrium} such that, for any any solution $S(\tau)=(\widetilde{\mathcal R}(\tau), \varrho(\tau), \widetilde{\mathcal S}(\tau), \sigma(\tau))$ of~\eqref{equations} approaching the equilibrium there exists a solution $S^{\rm c}(\tau)=(\widetilde{\mathcal R}^{\rm c}(\tau), \varrho^{\rm c}(\tau), \widetilde{\mathcal S}^{\rm c}(\tau), \sigma^{\rm c}(\tau)) \in{\mathcal M}^{\rm c}$ such that $|S(\tau)-S^{\rm c}(\tau)|$ is
an exponentially small remainder as $\tau\to+\infty$ and, moreover, 
\begin{eqnarray}\label{rho=0}\varrho^{\rm c}(\tau)\equiv 0\,,\qquad\widetilde{\mathcal R}^{\rm c}(\tau)\equiv\widetilde{\mathcal R}_0(\tau):=-\sqrt{2\Big(V(\sigma^{\rm c}(\tau))-T(\widetilde{\mathcal S}^{\rm c}(\tau), \sigma^{\rm c}(\tau))\Big)}\,.
\end{eqnarray}
	\end{proposition}
	To prove Proposition~\ref{flow on center manifold} we need some preliminary consideration. 
We denote as 
\begin{eqnarray}\label{E}E(\tau):=\frac{\widetilde{\mathcal R}(\tau)^2}{2}+{T(\widetilde{\mathcal S}(\tau), \sigma(\tau))}-{V(\sigma(\tau))}
\end{eqnarray}
the value of the numerator of~\eqref{total energy} along the solutions of the system~\eqref{systemNEW}.
\begin{lemma}\label{lem: energy dissipation}
Along the solutions of~\eqref{equations}, it is
\begin{eqnarray}\label{energy dissipation}E(\tau)=E(0)\,e^{\int_0^\tau\widetilde{\cal R}(s)ds}
\end{eqnarray}
\end{lemma}
\proof Let $\varrho(0)\ne 0$, so, by~\eqref{varrho(tau)}, $\varrho(\tau)\ne 0$ for all $\tau\in [0, +\infty)$.  The energy conservation gives
$$\frac{E(\tau)}{{\varrho(\tau)}}=\frac{E(0)}{{\varrho(0)}}$$
Combining with~\eqref{varrho(tau)} we have~\eqref{energy dissipation}.
Even though~\eqref{energy dissipation} has been obtained under condition $\varrho(0)\ne0$, it holds for all $\varrho(0)\in \mathbb R$, because  the vector--field in~\eqref{equations} is smooth for  all $\varrho\in \mathbb R$. 
$\qquad \square$

\proof {\bf of Proposition~\ref{flow on center manifold}} 
By the Center Manifold Theorem~\cite{bressan2007}, there exists a $k$--dimensional manifold $\widehat{\mathcal M}^{\rm c}$ tangent to $\mathbb E^{\rm c}$ at the equilibrium $S_{\star}=(\widetilde{\mathcal R}_{\star}, 0, 0, \sigma_{\star})$ defined in~\eqref{equilibrium},  such that for any  solution $S(\tau)=(\widetilde{\mathcal R}(\tau), \varrho(\tau), \widetilde{\mathcal S}(\tau), \sigma(\tau))$ of~\eqref{equations} approaching $S_\star$, 
there exists a solution $\widehat S^{\rm c}(\tau)$ $=$ $(\widetilde{\mathcal R}^{\rm c}(\tau)$, $\varrho^{\rm c}(\tau)$, $\widetilde{\mathcal S}^{\rm c}(\tau)$, $\sigma^{\rm c}(\tau))$ $\in$ $\widehat{\mathcal M}^{\rm c}$ approaching $S_\star$, such that $| S(\tau)-\widehat S^{\rm c}(\tau)|$ tends to zero exponentially fast.  Now, for any orbit
$\widehat S^{\rm c}(\tau)$ $=$ $(\widetilde{\mathcal R}^{\rm c}(\tau)$, $\varrho^{\rm c}(\tau)$, $\widetilde{\mathcal S}^{\rm c}(\tau)$, $\sigma^{\rm c}(\tau))$ $\in$ $\widehat{\mathcal M}^{\rm c}$, consider the solution $\tau\to  S^{\rm c}(\tau)$ of~\eqref{equations}  with initial datum
$ S^{\rm c}(0)=(\widetilde{\mathcal R}_0(0), 0, \widetilde{\mathcal S}^{\rm c}(0), \sigma^{\rm c}(0))$, with $\widetilde{\mathcal R}^{\rm c}_0(\tau)$ as in~\eqref{rho=0}.
Let ${\mathcal M}^{\rm c}$ be the collection of all such  $S^{\rm c}(\tau)$'s.  ${\mathcal M}^{\rm c}$ is trivially tangent to $\mathbb E^{\rm c}$ at $S_\star$.  We prove that, if $S^{\rm c}(\tau)$ corresponds to $\widehat S^{\rm c}(\tau)$ as said, then $|S^{\rm c}(\tau)-\widehat S^{\rm c}(\tau)|$  goes to zero exponentially fast, which will conclude the proof. 
For what concerns $S^{\rm c}(\tau)$, note that
 the choice of $\widetilde{\mathcal R}_0(0)$ implies
$E(\widetilde{\mathcal R}_0(0), \widetilde{\mathcal S}^{\rm c}(0), \sigma^{\rm c}(0))=0$.
Due to the split~\eqref{systemNEW}--\eqref{varrho(tau)}, and as -- by~\eqref{varrho(tau)} and~\eqref{energy dissipation} --   the manifolds $\{\varrho=0\}$ and $\{E=0\}$ are invariant, such orbit is given by
$S^{\rm c}(\tau)=(\widetilde{\mathcal R}_0(\tau), 0, \widetilde{\mathcal S}^{\rm c}(\tau), \sigma^{\rm c}(\tau))$, with $\widetilde{\mathcal R}_0(\tau)$ as in~\eqref{rho=0}. 
For what concerns $\widehat S^{\rm c}(\tau)$, remark that, by the definition~\eqref{E}, the following relation holds
\begin{eqnarray*}\widetilde{\mathcal R}^{\rm c}(\tau)\equiv-\sqrt{2\Big(E^{\rm c}(\tau)+V(\sigma^{\rm c}(\tau))-T(\widetilde{\mathcal S}^{\rm c}(\tau), \sigma^{\rm c}(\tau))\Big)}
\end{eqnarray*}
with $E^{\rm c}(\tau):=E(\widetilde{\mathcal R}^{\rm c}(\tau), \widetilde{\mathcal S}^{\rm c}(\tau), \sigma^{\rm c}(\tau))$.
As $$\lim_{\tau\to+\infty}\widetilde{\mathcal R}^{\rm c}(\tau)={\mathcal R}_\star<0$$
by~\eqref{varrho(tau)} and~\eqref{energy dissipation},
 $\varrho^{\rm c}(\tau)$  and  $E^{\rm c}(\tau)$ tend to zero exponentially fast. Then so does $|S^{\rm c}(\tau)-\widehat S^{\rm c}(\tau)|=O(|E^{\rm c}(\tau)|)+O(\varrho^{\rm c}(\tau))$.  $\qquad \square$ 
\vskip.1in
\noindent
We study the flow~\eqref{equations} on $\mathcal M^{\rm c}$. 
Ruling the motion of the coordinates $\varrho$ and $\mathcal R$ out through
the equalities~\eqref{rho=0}, we only need to study the system given by the two last equations in~\eqref{equations}, with $-\sqrt{2\Big(V(\sigma)-T(\widetilde{\mathcal S}, \sigma)\Big)}$ replacing $\widetilde{\mathcal R}$.
As we did for the study of the linearized system, we  switch to the coordinates $(s, w)$ defined via~\eqref{shift} and~\eqref{changeNEW}. We denote as
${\rm T}(w, s)$, ${\rm V}(s)$
the values of $T(\widetilde{\mathcal S}, \sigma)$, $V(\sigma)$ in the new coordinates, with
 $${\rm T}(w, s)=\frac{1}{2}w\cdot {\rm A}(s)w\,.$$
Then we arrive at
  \begin{eqnarray}\label{reduced equationsNEW}
 \left\{
 \begin{array}{lll}
 w'=
 -\frac{1}{2}w\cdot \partial_s{\rm A}(s)w+\partial_s{\rm V}(s)
 +\frac{w}{2}\sqrt{2{\rm V}(s)-
w\cdot {\rm A}(s)w
 	}\\\\
 \displaystyle s'={\rm A}(s)w
 \end{array}
 \right.
 \end{eqnarray}
By Lemma~\ref{diagonalization}, the projection of the center space onto the directions $(w, s)$  is 
$${\mathbb E}^{\rm c}_{\rm red}=\{0\}\times Ker(D)=\{0\}\times Span\{e_1, \ldots, e_k\}\,.$$
Here, without loss of generality, we have assumed that the first $k$ entries $c_1$, $\ldots$, $c_k$ on the principal diagonal of $D$ vanish, while $c_{k+1}$, $\ldots$, $c_{3n-3}$ are all different from zero.and we have denoted as $e_j$ is the unit vetor in the direction $j$ (E. g. $e_1=(1, 0, \ldots, 0)\in {\mathbb R}^{3n-3}$).
The projection  $\mathcal M_{\rm red}^{\rm c}$ of the center manifold $\mathcal M^{\rm c}$, tangent to ${\mathbb E}^{\rm c}_{\rm red}$ at $(0, 0)$,
is a graph whose projection on ${\mathbb E}^{\rm c}_{\rm red}$ is the identity. Therefore, it is a graph
 with respect to the first $k$ components of $s$. It is then convenient to split
\begin{eqnarray}\label{split}s=(x, y)
\end{eqnarray}
where $x$ 
are the first $k$ components of $s$; $y$ the last $3n-3-k$
. So, $\mathcal M_{\rm red}^{\rm c}$ has the form
\begin{eqnarray}\label{Mred}{\mathcal M}_{\rm red}^{\rm c}=\Big\{y=f(x)\,,\ w=\Phi(x)\Big\}\end{eqnarray}
By the tangency condition,
\begin{eqnarray}\label{conditions}f(0)=0\,,\qquad \Phi(0)=0\,,\qquad Df(0)=0\,,\qquad D\Phi(0)=0\,.\end{eqnarray}
We moreover let
 \begin{eqnarray}\label{definitions}g(x):=(x, f(x))\,,\qquad b(x):={\rm A}(g(x))\Phi(x)
 \end{eqnarray}
and, similarly to~\eqref{split}, we split
\begin{eqnarray}\label{phipsi}b(x)=(\phi(x), \psi(x))\end{eqnarray}
Note that the functions $\phi$, $\psi$  satisfy conditions~\eqref{conditions}, as well.
Using~\eqref{Mred},~\eqref{definitions} and~\eqref{phipsi} into~\eqref{reduced equationsNEW}, we obtain
\begin{eqnarray}\label{center manifold equations}
\left\{
\begin{array}{lll}
\displaystyle x'=\phi(x)\\\\
\displaystyle y'=Df(x)x'=\psi(x)\\\\
\displaystyle w'=D\Phi(x)x'=-\frac{1}{2}\Phi(x)\cdot \partial{\rm A}(g(x))\Phi(x)+\partial{\rm V}(g(x))+\frac{\Phi(x)}{2}\sqrt{2 {\rm V}(g(x))-\Phi(x)\cdot {\rm A}(g(x))\Phi(x)}
\end{array}
\right.
\end{eqnarray}
By construction, $\partial_s{\rm V}(0)=\alpha C \partial_\sigma V(\sigma_\star)=0$. The next Lemma quantifies the slope of the descent to $0$. 
\begin{lemma}\label{gradient bounds}
There exists a neighborhood $\mathcal U$ of $x=0$ such that 
\item[{\rm(i)}] there exist two positive numbers $0<k_1<k_2$ such that
$$k_1|\phi(x)|\le |\partial{\rm V}(g(x))|\le k_2|\phi(x)|\qquad \forall\ x\in {\mathcal U}\,;$$
\item[{\rm(ii)}] there exists $c>0$ such that, for all $T>0$ and for all  solutions $x(\tau)$ of~\eqref{center manifold equations} which does not leave $\mathcal U$ for $\tau>T$,
 \begin{eqnarray}\label{ii}\frac{d}{d\tau}{\rm V}\big(g(x(\tau))\big)\le-c|\partial{\rm V}\big(g(x(\tau))\big)|^2+o(|\partial{\rm V}\big(g(x(\tau))\big)|^2)\qquad \forall\ \tau>T\,.
\end{eqnarray}
\end{lemma}
\proof (i)\ 
We rewrite the latter equation in~\eqref{center manifold equations} as
\begin{eqnarray}\label{partialV}\partial{\rm V}(g(x))
=D\Phi(x)x'+\frac{1}{2}\Phi(x)\cdot \partial{\rm A}(g(x))\Phi(x)-\frac{\Phi(x)}{2}\sqrt{2 {\rm V}(g(x))-\Phi(x)\cdot {\rm A}(g(x))\Phi(x)}
\end{eqnarray}
We manage to bound all the terms at right hand side.
First, note that, combining the two first equations with the definitions of $b(x)$, $\phi(x)$, $\psi(x)$ in~\eqref{definitions},~\eqref{phipsi} allows to rewrite the function $\Phi(x)$  as linear functions of $\phi(x)$:
\begin{eqnarray}\label{Phi(x)}\Phi(x)=\big({\rm A}(g(x))\big)^{-1}b(x)=\big({\rm A}(g(x))\big)^{-1}\Big(\phi(x), Df(x)\phi(x)\Big)=\big({\rm A}(g(x))\big)^{-1}D g(x)\phi(x)
\end{eqnarray}
whence, $\Phi(x)=O(|\phi(x)|)$, so 
$$\frac{1}{2}\Phi(x)\cdot \partial{\rm A}(g(x))\Phi(x)=O(|\phi(x)|^2)\,.$$
Taking now the $\tau$--derivative of~\eqref{Phi(x)} and again using the former equation in~\eqref{center manifold equations} gives
\begin{eqnarray}\label{derivative}D\Phi(x)x'&=&\vec\partial_x\Big(\big({\rm A}(g(x))\big)^{-1}\Big)\vec\phi(x)\,\ \Big(\phi(x), Df(x)\phi(x)\Big)\nonumber\\
&&+\big({\rm A}(g(x))\big)^{-1}\Big(\vec\partial_x\phi(x)\vec \phi(x), \vec\partial_x\Big(Df(x)\phi(x)\Big)\vec\phi(x)\Big)
\end{eqnarray}
where we have used the notation
$$\vec\partial_\zeta a(\zeta)\vec b(\zeta):=\sum_j\frac{\partial a(\zeta)}{\partial\zeta_j}b_j(\zeta)\,.$$
The expression in~\eqref{derivative} shows that
$$D\Phi(x)x'=O(|\phi(x)|^2+o(|\phi(x)|)
=o(|\phi(x)|)$$
where the $o$ comes because
 $\phi(x)$ vanishes at $x=0$ together with its first derivatives.
 Finally, taking~\eqref{Phi(x)} into account, we have
 $$\frac{\Phi(x)}{2}\sqrt{2 {\rm V}(g(x))-\Phi(x)\cdot {\rm A}(g(x))\Phi(x)}=O(|\phi(x)|)$$
 Indeed,  as $x\to 0$, $g(x)\to 0$, so ${\rm V}(g(x))\to V(0)> 0$, while $\Phi(x)\to 0$. Therefore, the term inside the square root  keeps positive in a sufficiently small neighborhood of $x=0$.  Collecting all bounds into~\eqref{partialV}, we have 
 $$\partial{\rm V}(g(x))=O(|\phi(x)|)$$
 as claimed. \vskip.1in
 \noindent
 (ii)
 Combining the third equation in~\eqref{center manifold equations} with~\eqref{Phi(x)}, we have
 $$Dg(x) \phi(x)={\rm A}(g(x))\Phi(x)=-\frac{2{\rm A}(g(s))\partial{\rm V}(g(x))}{\sqrt{2 {\rm V}(g(x))-\Phi(x)\cdot {\rm A}(g(x))\Phi(x)}}+o(|\phi(x)|)$$
 But:
\begin{eqnarray*}2 {\rm V}(g(x))-\Phi(x)\cdot {\rm A}(g(x))\Phi(x)&=&2V(0)+\Big(2 {\rm V}(g(x))-2V(0)-\Phi(x)\cdot {\rm A}(g(x))\Phi(x)\Big)\nonumber\\
&=&2V(0)+O(|\partial({\rm V}\circ g)|)+O(|\phi^2|)\nonumber\\
&= &2V(0)+O(|\partial{\rm V}\circ g|)+O(|\phi^2|)\nonumber\\
&=&2V(0)+O(|\partial{\rm V}\circ g|)\end{eqnarray*}
where we have used 
$O(|\partial({\rm V}\circ g)|)=O(|\partial{\rm V}\circ g|)$,
as a consequence of
 the chain rule: $\partial({\rm V}(g(x)))=(D g)^{\rm t}(g(x))\partial{\rm V}(g(x))$.
Therefore, by (i),
\begin{eqnarray*}Dg(x) \phi(x)&=&-\frac{2{\rm A}(g(s))\partial{\rm V}(g(x))}{\sqrt{2 {\rm V}(0)}}+O(|\partial{\rm V}\circ g|^2)+o(|\phi(x)|)\nonumber\\
	&=&-\frac{2{\rm A}(g(s))\partial{\rm V}(g(x))}{\sqrt{2 {\rm V}(0)}}+o(|\partial{\rm V}(g(x))|)
	\end{eqnarray*}
  We use this bound into 
 \begin{eqnarray*}\frac{d}{d\tau}{\rm V}\big(g(x(\tau))\big)&=&
 \partial{\rm V}(g(x))\cdot Dg(x)\phi(x)\nonumber\\
 &=&-\frac{2
 	\partial{\rm V}(g(x))\cdot
 	{\rm A}(g(x)))\partial{\rm V}(g(x))}{\sqrt{2 {\rm V}(0)}}+o(|\partial{\rm V}(g(x))|^2)\nonumber\\
 &\le&-c|\partial{\rm V}(g(x))|^2+o(|\partial{\rm V}(g(x))|^2)
 \end{eqnarray*}
with a suitable $c>0$. Here, we have taken into account that ${\rm A}(s)$ is positive definite for all $s$, so $\zeta\to \zeta\cdot {\rm A}(s)\zeta$ has a positive minimum, uniformly with respect to $s$.
$\qquad \square$
\begin{lemma}\label{W} There exists a neighborhood ${\mathcal U}$ of $x=0$,  $c>0$ and $\alpha\in (1, 2)$ such that, for all $T>0$  and all solutions $x(\tau)$ of~\eqref{center manifold equations}  such that $x(\tau)\in {\mathcal U}$ for all $\tau >T$, the function
$W(\tau):={\rm V}\big(g(x(\tau))\big)-{\rm V}(0)$
satisfies
\begin{eqnarray}\label{integrability}0\le W(\tau)\le \frac{C}{\Big(\tau-\tau_0+1\Big)^{\frac{1}{\alpha-1}}}\qquad \forall\ \tau\ge \tau_0\ge T\end{eqnarray}
with
$C:=\max\left\{W(\tau_0)\,,\  \Big(c(\alpha-1)\Big)^{-\frac{1}{\alpha-1}}\right\}$.
\end{lemma}
\proof 
We  prove that 	there exists a neighborhood ${\mathcal U}$ of $x=0$,  $c>0$ and $\alpha\in (1, 2)$ such that, for all $T>0$  and all solutions $x(\tau)$ of~\eqref{center manifold equations}  such that $x(\tau)\in {\mathcal U}$ for all $\tau >T$, it is
	\begin{eqnarray}\label{differential inequality}{\rm V}\big(g(x(\tau))\big)-{\rm V}(0)\ge 0\,,\quad \frac{d}{d\tau}{\rm V}\big(g(x(\tau))\big)\le-c\Big({\rm V}\big(g(x(\tau))\big)-{\rm V}(0)\Big)^{\alpha}\quad \forall\ \tau>T\,.
	\end{eqnarray}
	Taking the integral in $[\tau_0, \tau]$, we shall have 
	$$W(\tau)\le \frac{W(\tau_0)}{\Big(
	1+c(\alpha-1)W(\tau_0)^{\alpha-1}(\tau-\tau_0)
	\Big)^{\frac{1}{\alpha-1}}}$$
	which implies~\eqref{integrability}.
Let $\mathcal U$ and $c$ be as in Lemma~\ref{gradient bounds}. The function $s\to {\rm V}(s)$ is holomorphic.
By the Lojasievicz inequality, there exist $C>0$ and $\theta\in (1/2, 1)$ and a small neighborhood $\mathcal V$ of $s=0$ such that
$$|{\rm V}(s)-{\rm V}(0)|^\theta\le C|\partial{\rm V}(s)|\qquad \forall\ s\in {\mathcal V}$$ 
Taking $s=g(x(\tau))$ and possibly smaller ${\mathcal U}\subset g^{-1}({\mathcal V})$ and using~\eqref{ii}, one has
\begin{eqnarray*}\cfrac{d}{d\tau}{\rm V}\big(g(x(\tau))\big)\le-\frac{c}{C}|{\rm V}\big(g(x(\tau))\big)-{\rm V}(0)|^{2\theta}
\end{eqnarray*}
Note that, as $|{\rm V}\big(g(x(\tau))\big)-{\rm V}(0)|$ is infinitesimal, we can assume it is smaller than $1$ for all $\tau>T$. Then we have
\begin{eqnarray}\label{decreasing}\cfrac{d}{d\tau}{\rm V}\big(g(x(\tau))\big)\le-\frac{c}{C}|{\rm V}\big(g(x(\tau))\big)-{\rm V}(0)|^{\alpha}
\end{eqnarray}
with $\alpha:=2\theta\in (1, 2)$.  Observe now that the function 
${\rm V}\big(g(x(\tau))\big)-{\rm V}(0)$ goes to $0$ as $\tau\to +\infty$ and,  by~\eqref{decreasing},  is decreasing for all $\tau\in [T, +\infty)$. Then ${\rm V}\big(g(x(\tau))\big)-{\rm V}(0)\ge 0$ for  all $\tau\in [T, +\infty)$. This is the former inequality in~\eqref{differential inequality}. Then the modulus  at right hand side of~\eqref{decreasing} can be neglected and, 
renaming $c/C$ as $c$, we have the latter.
$\quad \square$
\vskip.1in
\noindent
\proof {\bf of Lemma~\ref{finite Sint}}
Using Lemma~\ref{gradient bounds}, (ii)
\begin{eqnarray*}\int_a^b |\widetilde {\mathcal S}(\tau)|d\tau&\le& c_1\int_a^b |w(\tau)|d\tau=c_1\int_a^b |\Phi(x(\tau))|d\tau\le c_2\int_a^b |\phi(x(\tau))|d\tau\nonumber\\
&\le& c_3\int_a^b |\partial{\rm V}(g(x(\tau)))|d\tau\le 
 c_3\left(\int_a^b |\partial{\rm V}(g(x(\tau)))|^2d\tau\right)^{\frac{1}{2}}
 \left(\int_a^b d\tau\right)^{\frac{1}{2}}\nonumber\\
 &\le&c_4(b-a)^{\frac{1}{2}}\left(\int_a^b 
\left(-\frac{d}{d\tau}{\rm V}(g(x(\tau)))\right)
 d\tau\right)^{\frac{1}{2}}\nonumber\\
 &=&c_4(b-a)^{\frac{1}{2}} \Big(\big({\rm V}(g(x(a)))-{\rm V}(0)\big)-\big({\rm V}(g(x(b)))-{\rm V}(0)\big)\Big)^{\frac{1}{2}}
\nonumber\\
 &\le& c_4(b-a)^{\frac{1}{2}} \Big({\rm V}(g(x(a)))-{\rm V}(0)\Big)^{\frac{1}{2}}
\end{eqnarray*}
because $W(\tau)={\rm V}(g(x(\tau)))-{\rm V}(0)$ is positive for all $\tau\in [0, +\infty)$. Using now Lemma~\ref{W}
\begin{eqnarray*}\int_a^b |\widetilde {\mathcal S}(\tau)|d\tau&\le& c_5(b-a)^{\frac{1}{2}} a^{-\frac{1}{2(\alpha-1)}}
\end{eqnarray*}
With $b=2a$,
\begin{eqnarray*}\int_a^{2a} |\widetilde {\mathcal S}(\tau)|d\tau&\le& c_5 a^{-\gamma}
\end{eqnarray*}
with $\gamma:=\frac{2-\alpha}{2(\alpha-1)}>0$. It follows
\begin{eqnarray*}\int_T^{+\infty} |\widetilde {\mathcal S}(\tau)|d\tau&=&\sum_{k=0}^{+\infty}\int_{2^kT}^{2^{k+1}T}|\widetilde {\mathcal S}(\tau)|d\tau\le c_5\sum_{k=0}^{+\infty} (2^kT)^{-\gamma}
\nonumber\\
&=&\frac{c_5 T^{-\gamma}}{1-2^{-\gamma}}<+\infty\,.\qquad \square
\end{eqnarray*}

\appendix

\section{Barycentric reduction (generalized Jacobi coordinates)}
	\label{Barycentric reduction (generalized Jacobi coordinates)}
 Let  us introduce the coordinates $B$, $x_1$, $\ldots$, $x_{n}\in {\mathbb R}^3$ via the linear change 
\begin{eqnarray}\label{xi}
 \left\{\begin{array}{lll}
 \displaystyle x_i:=\sum_{j=1}^i \frac{m_j q_j}{M_i}-q_{i+1}\qquad &i=1\,,\ldots\,, n \\\\
  \displaystyle B:=\sum_{j=1}^{n+1} \frac{m_j q_j}{M_{n+1}}
 \end{array}
 \right.\end{eqnarray}
where
  $$M_0:=0\,,\qquad M_i:=\sum_{j=1}^i
 m_i\,.$$
Namely, $B$ is the center of mass of the whole system, while, if $i=1$, $\ldots$, $n$,  $x_i$ is the relative distance of the particle with mass $m_{i+1}$  to the center of mass of $m_{1}$, $\ldots$, $m_{i}$. The inverse formula of~\eqref{xi} is 
\begin{eqnarray}\label{qi}q_i=\left\{
 \begin{array}{lll}
 \displaystyle-\frac{M_{i-1}}{M_i}x_{i-1}+\sum_{k=i}^{n}\frac{m_{k+1}}{M_{k+1}}x_{k}+B\quad &i=1\,,\ldots n\\\\
  \displaystyle-\frac{M_{n}}{M_{n+1}}x_{n}+B & i=n+1
 \end{array}
 \right.
 \end{eqnarray}
The  generalized impulses $y_1$, $\ldots$, $y_n$, $P$,  conjugated to $x_1$, $\ldots$, $x_n$, $B$, are defined through  the identity
 $$PdB+\sum_{i=1}^{n} y_i\cdot dx_i=\sum_{i=1}^{n+1} p_i\cdot dq_i\,.$$
We easily find the relation
 \begin{eqnarray}\label{yi}
 p_i
 =m_i\sum_{j=i}^{n+1} \frac{y_j}{M_j}-y_{i-1}
 \qquad i=1\,,\ldots n+1
 \end{eqnarray}
 with $y_0:=0$ $y_{n+1}:=P$.
We observe that $P$ is the total linear momentum, a first integral, for which reason $B$, the position of the center of mass, its conjugated coordinate, will be ignorable in ${\rm H}_{n+1}$. Indeed,
 $$\sum_{i=1}^{n+1} p_i=\sum_{i=1}^{n+1}\left(\frac{m_i}{M_{n+1}}P+ (1-\delta_{i, n+1}) m_i\sum_{j=i}^{n} \frac{y_j}{M_j}-y_{i-1}\right)=P+\sum_{i=1}^{n}\sum_{j=i}^n m_i\frac{y_j}{M_j}-\sum_{i=1}^n y_i=P$$
 having interchanged $\sum_{i=1}^{n}\sum_{j=i}^n=\sum_{j=1}^n\sum_{i=1}^{j}$ and renamed $i-1\to i$.
 \subsection*{Translationally reduced Hamiltonian}
  We now aim to use the coordinates $(y_i, x_i)$ into the Hamiltonian~\eqref{H}.
 We start with the potential energy $V$, which, as remarked, will be $B$--independent, hence will depend on $x$ only.
 From~\eqref{qi} we get, for all $1\le i<j\le n+1$,
 $$ q_i-q_{j}=\sum_{k=i}^{j-1} (q_i-q_{i+1})
 =-\frac{M_{i-1}}{M_i}x_{i-1}+
 \sum_{k=i}^{j-2}\frac{m_{k+1}}{M_{k+1}}x_k+
 x_{j-1}$$
 where the summand is to be neglected if $j-i=1$.
 Then, renaming $i-1\to i$, $j-1\to j$,
 \begin{eqnarray}\label{V}V(x)=-\sum_{0\le i<j\le n}\frac{m_{i+1}m_{j+1}}{\left|-\frac{M_{i}}{M_{i+1}}x_{i}+
 \sum_{k=i+1}^{j-1}\frac{m_{k+1}}{M_{k+1}}x_k+
 x_{j}
 \right|} \end{eqnarray}
 with $M_0:=0$.
 Now we look at the kinetic energy. Using~\eqref{yi}, we have
   \begin{eqnarray*}
 T&=&\sum_{j=1}^{n+1}\frac{1}{2m_j}\left|m_j\sum_{i=j}^{n+1} \frac{y_i}{M_i}-y_{j-1}\right|^2\nonumber\\
 &=&\frac{|y_{n+1}|^2}{2M_{n+1}}+\sum_{1\le i\le n}\left(
 \frac{1}{{M}_i}+\frac{1}{m_{i+1}}\right)|y_i|^2\nonumber\\
 &+&\sum_{j=1}^{n+1}m_j\sum_{j\le i<k\le n+1}\frac{y_i\cdot y_k}{M_i M_k}-\sum_{j=2}^{n+1}\sum_{i=j}^{n+1} \frac{y_i\cdot y_{j-1}}{M_i}
 \end{eqnarray*}
 where, as above, $y_{n+1}:=P$.
The main fact is now  that
 the ``rectangular''  in
 namely, the last line above,  vanish. Indeed, the first summand can be written as
 $$\sum_{1\le i<k\le n+1}\frac{y_i\cdot y_k}{M_i M_k}\sum_{j=1}^i m_j= \sum_{1\le i<k\le n+1}\frac{y_i\cdot y_k}{M_i M_k}M_i= \sum_{1\le i<k\le n+1}\frac{y_i\cdot y_k}{M_k}$$
while the latter as
 $$\sum_{j=2}^{n+1}\sum_{i=j}^n \frac{y_i\cdot y_{j-1}}{M_i}=\sum_{i=1}^{n+1}\sum_{k=i+1}^{n+1} \frac{y_i\cdot y_{k}}{M_k}= \sum_{1\le i<k\le n+1}\frac{y_i\cdot y_k}{M_k}$$
 The two summands are the same, hence the rectangular terms vanish, hence only the squared ones survive:
  \begin{eqnarray}\label{T}
 T=\frac{|P|^2}{2M_{n+1}}+\sum_{1\le i\le n}\frac{1}{2}\left(
 \frac{1}{{M}_i}+\frac{1}{m_{i+1}}\right)|y_i|^2=\sum_{i=1}^{n+1}\frac{|y_i|^2}{2\mu_i}
 \end{eqnarray}
 with
\begin{eqnarray}\label{mui}\mu_i:=\left\{
 \begin{array}{lll}
\displaystyle \frac{m_{i+1}M_i}{M_{i+1}}\qquad &i=1\,,\ldots\,,n\\\\
\displaystyle  M_{n+1} &i=n+1
 \end{array}
 \right.
 \end{eqnarray}
 Picking~\eqref{V} and~\eqref{T} we see that, after the change defined in~\eqref{qi},~\eqref{yi},
  $H_{n+1}(y, x)$ splits as in~\eqref{Hn-1}
\begin{eqnarray}\label{Hn-1}H_{n+1}=H_{B}(P)+ H(y,  x)\end{eqnarray}
 with $y=(y_1, \ldots, y_{n})$, $ x=(x_1, \ldots, x_{n})$, with
 $$H_{B}(P)=\frac{|P|^2}{2M_{n+1}}$$
 while $H( y, x)$ as in~\eqref{HNEW}.

\section*{Proof of~\eqref{eq: ang mom}}
The equality~\eqref{eq: ang mom} is a consequence of  the canonical character of the transformation~\eqref{reduction of translations} and  the fact that it acts such in a way that 
the components $p_{i, j}$, $q_{i, j}$
of
 $p_i$ and $q_i$,  $i=1$, $\ldots$, $n+1$; $j=1$, $2$, $3$,
 are related through
\begin{eqnarray}\label{linear combinations}q_{i, j}={\mathcal L}_i(B_j, x_{\cdot, j})\,,\qquad p_{i, j}=\widetilde{\mathcal L}_i(P_j, y_{\cdot, j})\qquad \forall\ i=1\,,\ldots\,,n+1\,,\qquad j=1, 2, 3\,.
\end{eqnarray}
where, for $i=1$, $\ldots$, $n+1$,  
\begin{eqnarray*}q_i={\mathcal L}_i(B, x)\,,\qquad p_i=\widetilde{\mathcal L}_i(P, y)\end{eqnarray*}
are the linear combinations 
  defined via   formulae~\eqref{qi},~\eqref{yi}.
Indeed,  according to the conservation of the canonical $2$--form, if $({ p}, { q}) $, $ ({ p}', { q}')$ are the respective images of
$(P, y, B, x)$, $(P', y', B', x')$
through~\eqref{reduction of translations}, we have
\begin{eqnarray}\label{2form}
\sum_{i=1}^{n+1} \sum_{j=1}^3 \big({ p}_{i, j} { q}_{i, j}'-{ q}_{i, j} { p}_{i, j}'\big)=\sum_{i=1}^{n} \sum_{j=1}^3 \big({ y}_{i, j} { x}_{i, j}'-{ x}_{i, j} { y}_{i, j}'\big)+\sum_{j=1}^3\big(P_{j}B'_j-B_j P'_j\big)
\end{eqnarray}
We now choose  $(P', y', B', x')$  from $(P, y, B, x)$  just swapping the $j$--indices  $1$ and $2$. Namely, we take:
\begin{eqnarray*}
{ y}'_{i,j}&=&\left\{
\begin{array}{lll}
{ y}_{i, 2}\quad& j=1\\
{ y}_{i, 1}\quad& j=2\\
{ y}_{i, 3}\quad& j=3
\end{array}
\right.\qquad { x}'_{i,j}=\left\{
\begin{array}{lll}
{ x}_{i, 2}\quad& j=1\\
{ x}_{i, 1}\quad& j=2\\
{ x}_{i, 3}\quad& j=3
\end{array}
\right.\qquad i=1\,,\ldots, n\nonumber\\
{ P}'_{i,j}&=&\left\{
\begin{array}{lll}
{ P}_{i, 2}\quad& j=1\\
{ P}_{i, 1}\quad& j=2\\
{ P}_{i, 3}\quad& j=3
\end{array}
\right.\qquad { B}'_{i,j}=\left\{
\begin{array}{lll}
{ B}_{i, 2}\quad& j=1\\
{ B}_{i, 1}\quad& j=2\\
{ B}_{i, 3}\quad& j=3
\end{array}
\right.
\end{eqnarray*}
By the equalities~\eqref{linear combinations}, we have the same swap for the images:
\begin{eqnarray*}
{ p}'_{i,j}=\left\{
\begin{array}{lll}
{ p}_{i, 2}\quad& j=1\\
{ p}_{i, 1}\quad& j=2\\
{ p}_{i, 3}\quad& j=3
\end{array}
\right.\qquad { q}'_{i,j}=\left\{
\begin{array}{lll}
{ q}_{i, 2}\quad& j=1\\
{ q}_{i, 1}\quad& j=2\\
{ q}_{i, 3}\quad& j=3
\end{array}
\right.\qquad i=1\,,\ldots, n+1
\end{eqnarray*}
Using these expressions into~\eqref{2form} we have the third component of~\eqref{eq: ang mom}. The others are similarly reached.

\section{Euler angles}
\label{sec:Euler-angles}

Assume that $O{e}_1{e}_2{e}_3$ is the standard coordinate frame in $\mathbb{R}^3$ and $O{f}_1{f}_2{f}_3$ be other orthogonal coordinate frame.
Assume 
\begin{equation*}
   e_3 \not\parallel\  f_3
\end{equation*}
We denote as $\varphi$, $\theta$ and $\psi$  the precession, nutation and proper rotation angles relatively to the frames $Oe_1e_2e_3$ and $O{f}_1{f}_2{f}_3$, defined as follows. Given $u$, $v$, $w\in {\mathbb R}^3$, with $w\perp u$, $v$, let $\alpha_w(u, v)$ stand for the angle  which $u$ has to run, positively (counterclockwise with respect to $w$), to overlap its direction  and verse to the ones of $v$. Then
$$\varphi:=\alpha_{e_3}(e_1, \gamma)\,,\qquad \theta:=\alpha_{\gamma}(e_3, {f}_3)\,,\quad \psi:=\alpha_{{f}_3}(\gamma, {f}_1)$$
with
$\gamma:=\frac{e_3\times {f}_3}{|e_3\times {f}_3|}$ and
$$\varphi\,,\psi \mod\ 2\pi\,,\qquad \theta \in (0,\pi) $$
Our goal is to establish the following lemma which gathers all facts needed.

\begin{lemma}
\label{lem:EulerAngles}
Vectors $f_i$  are given by (i.e. $f_j$ are columns of the matrix on the right hand side of~\eqref{eq:appf-f1f2f3})
\begin{equation}
  \left({f}_1, {f}_2, {f}_3 \right)=\mathcal{R}_3(\varphi) \mathcal{R}_1(\theta)\mathcal{R}_3(\psi) ,  \label{eq:appf-f1f2f3}
\end{equation}
where 
\begin{eqnarray*}
{\cal R}_1(\alpha)&=&\left(
\begin{array}{ccc}
1&0&0\\
0&\cos\alpha&-\sin\alpha\\
0&\sin\alpha&\cos\alpha
\end{array}
\right)\,,\quad  {\cal R}_3(\alpha)=\left(
\begin{array}{ccc}
\cos\alpha&-\sin\alpha&0\\
\sin\alpha&\cos\alpha&0\\
0&0&1
\end{array}
\right).\end{eqnarray*}
Moreover, we have
\begin{eqnarray}
   \mathcal{R}_3(\varphi) \mathcal{R}_1(\theta)e_3&=&f_3,   \label{eq:appR3R1e3f3}\\
   \mathcal{R}_3(\varphi) \mathcal{R}_1(\theta)e_1&=&\gamma. \label{eq:appR3R1e1g}
\end{eqnarray}
\end{lemma}

\begin{remark}\rm
From (\ref{eq:appf-f1f2f3}) the following explicit formulae for the $f_i$'s and $\gamma$ follow, which may turn to be useful in applications:
\begin{eqnarray*}
  \left( f_1,f_2,f_3 \right)=
  \left(
   \begin{smallmatrix}         \cos(\varphi) \cos(\psi) - \sin(\varphi)\cos(\theta)\sin(\psi), &  -\cos(\varphi) \sin(\psi) - \sin(\varphi)\cos(\theta)\cos(\psi), & \sin(\varphi)\sin(\theta) \\
            \sin(\varphi) \cos(\psi) + \cos(\varphi)\cos(\theta)\sin(\psi), & -\sin(\varphi) \sin(\psi) + \cos(\varphi)\cos(\theta)\cos(\psi), & -\cos(\varphi)\sin(\theta) \\
            \sin(\theta)\sin(\psi), & \sin(\theta)\cos(\psi), & \cos(\theta) 
           \end{smallmatrix}
           \right)
\end{eqnarray*}
From the formulas for $f_j$ obtained above we have (recall that $\sin \theta >0$ because $\theta \in (0,\pi)$)
\begin{equation*}
  \gamma=\left(\begin{array}{c}\cos \varphi \\ \sin \varphi \\ 0 \end{array} \right).
\end{equation*}
\end{remark}

\proof
For an nonzero vector  $w \in \mathbb{R}^3$ and $\varphi \in \mathbb{R}$ 
let us denote by $\mathcal{R}(w,\varphi)$ a rotation by angle $\varphi$ around axis $a$ in positive direction (counterclockwise with respect to $w$, hence  $\mathcal{R}(w,\varphi)$ preserves the orientation). Obviously $\mathcal{R}_3(\varphi)=\mathcal{R}(e_3,\varphi)$ and $\mathcal{R}_1(\varphi)=\mathcal{R}(e_1,\varphi)$.
Observe that the application of the following sequence of rotations $\mathcal{R}(e_3,\varphi)$, $\mathcal{R}(\gamma,\theta)$ and $\mathcal{R}(f_3,\psi)$ will map the frame $O{e}_1{e}_2{e}_3$ to the frame 
$O{f}_1{f}_2{f}_3$, i.e.   we have
\begin{equation*}
\mathcal{R}(f_3,\psi) \mathcal{R}(\gamma,\theta)  \mathcal{R}(e_3,\varphi) = \left(f_1,f_2,f_3\right)
\end{equation*}
Indeed, 
\begin{itemize}
\item for $\mathcal{R}(e_3,\varphi) $  from definition of $\varphi$ we have $e_3 \mapsto e_3$ and $e_1 \mapsto \gamma$ 
\item for $\mathcal{R}(\gamma,\theta)  \mathcal{R}(e_3,\varphi)$ from the above and the definition of $\theta$  we have $e_1 \mapsto \gamma$ and $e_3 \mapsto f_3$
\item for $\mathcal{R}(f_3,\psi) \mathcal{R}(\gamma,\theta)  \mathcal{R}(e_3,\varphi) $ from the above and the definition of $\psi$ we have $e_3 \mapsto f_3$ and $e_1 \mapsto f_1$. This implies (as the consequence 
of the preservation of the orientation) that also $e_2 \mapsto f_2$.
\end{itemize}
To finish the proof of (\ref{eq:appf-f1f2f3}) we need to show that $$\mathcal{R}(f_3,\psi) \mathcal{R}(\gamma,\theta)  \mathcal{R}(e_3,\varphi) =\mathcal{R}_3(\varphi) \mathcal{R}_1(\theta)\mathcal{R}_3(\psi).$$
Observe that $\mathcal{R}(\gamma,\theta) = T^{-1}\mathcal{R}_1(\theta) T$, where $T$ is any orthogonal orientation preserving transformation such that $\gamma  \mapsto e_1$. From the previous discussion 
regarding $\mathcal{R}(e_3,\varphi)$ it follows that we can take $T^{-1}=\mathcal{R}(e_3,\varphi)=\mathcal{R}_{3}(\varphi)$. Hence
\begin{eqnarray*}
 \mathcal{R}(\gamma,\theta)&=& \mathcal{R}_{3}(\varphi) \mathcal{R}_1(\theta) \left(\mathcal{R}_{3}(\varphi)\right)^{-1},  
 \end{eqnarray*}
 and 
 \begin{eqnarray}
  \mathcal{R}(\gamma,\theta)  \mathcal{R}(e_3,\varphi) &=& \mathcal{R}_{3}(\varphi) \mathcal{R}_1(\theta) .  \label{eq:RgthetaRe3varphi}
\end{eqnarray}
Analogously, $\mathcal{R}(f_3,\psi) =T^{-1}\mathcal{R}_3(\psi)T$, where $T$ is any orthogonal orientation preserving transformation such that $f_3 \mapsto e_3$. From the above discussion it follows that 
we can take $T^{-1}=\mathcal{R}(\gamma,\theta)  \mathcal{R}(e_3,\varphi)=\mathcal{R}_{3}(\varphi) \mathcal{R}_1(\theta)$. Therefore we have
\begin{equation*}
   \mathcal{R}(f_3,\psi) =\left(\mathcal{R}_{3}(\varphi) \mathcal{R}_1(\theta) \right) \mathcal{R}_3(\psi) \left(\mathcal{R}_{3}(\varphi) \mathcal{R}_1(\theta) \right)^{-1}
\end{equation*}
and finally we obtain
\begin{eqnarray*}
  \mathcal{R}(f_3,\psi) \mathcal{R}(\gamma,\theta)  \mathcal{R}(e_3,\varphi) &=& \left(\mathcal{R}_{3}(\varphi) \mathcal{R}_1(\theta) \right) \mathcal{R}_3(\psi) \left(\mathcal{R}_{3}(\varphi) \mathcal{R}_1(\theta) \right)^{-1}
  \left(\mathcal{R}_{3}(\varphi) \mathcal{R}_1(\theta) \right) = \\
  &=& \mathcal{R}_{3}(\varphi) \mathcal{R}_1(\theta)  \mathcal{R}_3(\psi) 
\end{eqnarray*}
This finishes the proof of (\ref{eq:appf-f1f2f3}). 
To obtain (\ref{eq:appR3R1e3f3},\ref{eq:appR3R1e1g})  observe that it follows from (\ref{eq:RgthetaRe3varphi}) and our discussion about the action of $\mathcal{R}(\gamma,\theta)  \mathcal{R}(e_3,\varphi)$. $\qquad \square$

\addcontentsline{toc}{section}{References}
 \bibliographystyle{plain}

 \end{document}